\documentclass{amsart}

\usepackage{amssymb}
\usepackage{amsthm}
\usepackage{enumitem}

\theoremstyle{plain}\newtheorem{Theorem}{Theorem}[section]
\theoremstyle{plain}\newtheorem{Conjecture}[Theorem]{Conjecture}

\theoremstyle{plain}\newtheorem{Corollary}[Theorem]{Corollary}
\theoremstyle{plain}\newtheorem{Lemma}[Theorem]{Lemma}
\theoremstyle{plain}
\theoremstyle{definition}
\theoremstyle{definition}
\theoremstyle{definition}
\theoremstyle{definition}\newtheorem{Remark}[Theorem]{Remark}
\theoremstyle{definition}\newtheorem{Notation}[Theorem]{Notation}
\theoremstyle{definition}
\theoremstyle{definition}
\theoremstyle{definition}
\theoremstyle{definition}\newtheorem{Contents}[Theorem]{Contents}
\theoremstyle{definition}
\theoremstyle{definition}
\theoremstyle{definition}\newtheorem{Notation/Definition}[Theorem]{Notation/Definition}
\def\C{{\mathbb C}}

\def\Q{{\mathbb Q}}

\def\Aut{\mathrm{Aut}}

\def\dim{\mathrm{dim}}

\def\Hom{\mathrm{Hom}}

\def\Inn{\mathrm{Inn}}
\def\Irr{\mathrm{Irr}}

\def\mod{\mathrm{mod}}

\newcommand{\GL}{\operatorname{GL}}

\begin{document}

\begin{flushleft}
{\sf May 21, 2012}
\end{flushleft}

\bigskip
\bigskip
\bigskip

\begin{center}
   {\Large\bf  
   Brou{\'e}'s abelian defect group conjecture holds for \\ \vspace*{0.3em} 
   the double cover 
   of the Higman-Sims sporadic simple group}
\end{center}

\bigskip
\bigskip
\bigskip

\begin{center} {\large
        {\bf Shigeo Koshitani}$^{\text a,*}$, 
        {\bf J{\"u}rgen M{\"u}ller}$^{\text b}$,
        {\bf Felix Noeske}$^{\text c}$  
} \end{center}

\bigskip

\begin{center} {\it
${^{\mathrm{a}}}$Department of Mathematics, Graduate School of Science, \\
Chiba University, Chiba, 263-8522, Japan \\
${^{\mathrm{b, \, c}}}$Lehrstuhl D f{\"u}r Mathematik,
RWTH Aachen University, 52062 Aachen, Germany}
\end{center}

\footnote
{$^*$ Corresponding author. \\
\indent {\it E-mail addresses:} koshitan@math.s.chiba-u.ac.jp (S.~Koshitani), \\
juergen.mueller@math.rwth-aachen.de (J.~M{\"u}ller),
Felix.Noeske@math.rwth-aachen.de (F.~Noeske). }

\bigskip

\begin{center}
{\large\it Dedicated to the memory of Herbert Pahlings }
\end{center}

\bigskip

\hrule 

\bigskip

{\small\noindent{\bf Abstract}

\bigskip\noindent
In the representation theory of finite groups, there is a well-known and
important conjecture, due to Brou{\'e}', saying that for any prime $p$, 
if a $p$-block $A$ of a finite group $G$ has an abelian defect group $P$, 
then $A$ and its Brauer corresponding block $B$ of the normaliser $N_G(P)$ 
of $P$ in $G$ are derived equivalent.
We prove in this paper, that Brou{\'e}'s abelian defect group conjecture,
and even Rickard's splendid equivalence conjecture
are true for the faithful $3$-block $A$ with an elementary abelian
defect group $P$ of order $9$ of the double cover 2.{\sf HS}
of the Higman-Sims sporadic simple group.
It then turns out that both conjectures hold
for all primes $p$ and for all $p$-blocks of 2.{\sf HS}.

 }

\bigskip

\hrule

\bigskip
\bigskip
\bigskip

\section{Introduction and notation}\label{intro}

\noindent
In the representation theory of finite groups, one of the
most important and interesting problems
is to give an affirmative answer to a conjecture
which was introduced by Brou{\'e} around 1988
\cite{Broue1990}.
He actually conjectures the following, where 
the various notions of equivalences used are recalled more
precisely in {\bf\ref{NotationEquivalences}}:

\begin{Conjecture} 
[Brou\'e's Abelian Defect Group Conjecture \cite{Broue1990}]
\label{ADGC}
Let $(\mathcal K, \mathcal O, k)$ be a
splitting $p$-modular system, where $p$ is a prime,
for all subgroups of a finite group $G$. 
Assume that $A$ is a block algebra of
$\mathcal OG$ with a defect group $P$ and that $A_N$ is a
block algebra of $\mathcal ON_G(P)$ such that $A_N$ is the
Brauer correspondent of $A$, where $N_G(P)$ is the normaliser of
$P$ in $G$. Then $A$ and $A_N$ should be derived equivalent
provided $P$ is abelian.
\end{Conjecture}

\noindent
In fact, a stronger conclusion than {\bf\ref{ADGC}} is expected:

\begin{Conjecture}
[Rickard's Splendid Equivalence Conjecture \cite{Rickard1996,Rickard1998}]
\label{RickardConjecture}
Keeping the notation, we suppose that $P$ is abelian as in {\bf\ref{ADGC}}.
Then there should be a splendid derived equivalence between
the block algebras $A$ of $\mathcal OG$ and $A_N$ of $\mathcal ON_G(P)$.
\end{Conjecture}

\noindent
There are several cases where the conjectures
{\bf\ref{ADGC}} and {\bf\ref{RickardConjecture}}
have been verified,
albeit the general conjecture is widely open;
for an overview, containing suitable references, 
see \cite{ChuangRickard}. As for general results
concerning blocks with a fixed defect group,
by \cite{Linckelmann1991,Rickard1989,Rouquier1995,Rouquier1998} 
conjectures {\bf\ref{ADGC}} and {\bf\ref{RickardConjecture}} are
proved for blocks with cyclic defect groups in arbitrary 
characteristic.

Moreover, in \cite[(0.2)Theorem]{KoshitaniKunugi2002} it is shown
that {\bf\ref{ADGC}} and {\bf\ref{RickardConjecture}} are true for 
the principal block algebra of an arbitrary finite group 
when the defect group is elementary abelian of order $9$.
In view of the strategy used in \cite{KoshitaniKunugi2002},
and of a possible future theory reducing {\bf\ref{ADGC}} and 
{\bf\ref{RickardConjecture}} to the quasi-simple groups,
it seems worth-while to proceed with this class of groups,
as far as non-principal $3$-blocks with elementary abelian 
defect group of order $9$ are concerned. Indeed, for these cases 
there are partial results already known, see
\cite{KessarLinckelmann2010,KoshitaniKunugiWaki2002,
      KoshitaniKunugiWaki2004,KoshitaniKunugiWaki2008,
      KoshitaniMueller2010,Kunugi,MuellerSchaps2008}
for instance.
The present paper is another step in that programme, 
our main result being the following:

\begin{Theorem}\label{MainTheorem}
Let $G$ be the double cover $2.{\sf HS}$
of the Higman-Sims sporadic simple group, and 
let $(\mathcal K, \mathcal O, k)$ be a splitting $3$-modular
system for all subgroups of $G$.
Suppose that $A$ is the faithful block algebra of $\mathcal OG$
with elementary abelian defect group $P = C_3 \times C_3$ of order $9$,
and that $B$ is a block algebra
of $\mathcal ON_G(P)$ such that $B$ is the Brauer correspondent
of $A$. Then, $A$ and $B$ are splendidly derived
equivalent, hence the conjectures 
{\bf\ref{ADGC}} and {\bf\ref{RickardConjecture}} of
Brou{\'e} and Rickard hold.
\end{Theorem}

\noindent
As an immediate corollary we get:

\begin{Corollary}\label{ADGCfor2HSforAllPrimes}
Brou\'e's abelian defect group conjecture {\bf\ref{ADGC}}, and 
even Rickard's splendid equivalence conjecture 
{\bf\ref{RickardConjecture}} are true
for all primes $p$ and for all block algebras of $\mathcal OG$.
\end{Corollary}

\noindent
Our strategy to prove {\bf\ref{MainTheorem}}
is similar to the ones pursued, for example,
for the Janko sporadic simple group {\sf J}$_4$ in 
\cite[1.6.Theorem]{KoshitaniKunugiWaki2008}
or the Harada-Norton sporadic simple group {\sf HN} 
in \cite[1.3.Theorem]{KoshitaniMueller2010}.
Our starting point was actually to realise that
the $3$-decomposition matrix of $A$
coincides (up to a suitable order of rows and columns)
with the $3$-decomposition matrix of the principal $3$-block 
$A'$ of the alternating group $\mathfrak A_8$ on $8$ letters:

\medskip
\begin{center}
\begin{tabular}{r|r|ccccc}
$A$ & $A'$ & & & & & \\
\hline
   $176$ \hspace*{0.15em} & $1$  & 1 & . & . & . & .\\
   $176^*$                & $7$  & . & 1 & . & . & .\\
   $616$ \hspace*{0.15em} & $14$ & 1 & . & 1 & . & .\\
   $616^*$                & $20$ & . & 1 & 1 & . & .\\
    $56$ \hspace*{0.15em} & $28$ & . & . & . & 1 & .\\
  $1000$ \hspace*{0.15em} & $35$ & . & . & . & . & 1\\
  $1792$ \hspace*{0.15em} & $56$ & 1 & 1 & 1 & . & 1\\
  $1232$ \hspace*{0.15em} & $64$ & 1 & . & . & 1 & 1\\
  $1232^*$                & $70$ & . & 1 & . & 1 & 1\\
\end{tabular}
\end{center}

\medskip\noindent
Here, we indicate ordinary irreducible characters 
just by their degrees, and complex conjugation by ${}^*$.
Therefore, it is quite natural to suspect 
that the block algebras $A$ and $A'$ are Morita equivalent, 
or even Puig equivalent. If this were true, then 
since conjectures {\bf\ref{ADGC}} and {\bf\ref{RickardConjecture}}
have been solved for $A'$ in \cite{Okuyama1997,Okuyama2000}, 
this would immediately entail their validity for $A$ as well.
Indeed, we are able to prove:

\begin{Theorem}\label{2HSandA8}
We keep the notation and the assumptions as in {\bf\ref{MainTheorem}},
and let $G' = \mathfrak A_8$ be the alternating group on $8$ letters.
Then, the block algebra $A$ of $\mathcal OG$
and the principal block algebra $A'$ of $\mathcal OG'$ 
are Puig equivalent.
\end{Theorem}

\begin{Remark}\label{twoequivalences}
A few remarks on {\bf\ref{2HSandA8}} are appropriate:

(a)
In order to prove {\bf\ref{2HSandA8}} in its full strength,
the detailed local analysis leads to a problem similar to the
one already encountered in \cite[6.14.Question]{KoshitaniKunugiWaki2008}:
Viewing $G$ and $G'$ as (unrelated) abstract groups would
only allow to prove that the block algebras $A$ and $A'$
are Morita equivalent, but not necessarily Puig equivalent.
In its consequence this would mean
that we were only able to verify Brou{\'e}'s conjecture {\bf\ref{ADGC}},
but not Rickard's conjecture {\bf\ref{RickardConjecture}} for $G$.
To remedy this, and to circumvent 
\cite[6.14.Question]{KoshitaniKunugiWaki2008},
we use the fact that $G'$ can be embedded as a subgroup into $G$, 
leading to an explicit configuration of groups
allowing for compatible local analysis.

\medskip
(b)
Note that complex conjugation induces a non-trivial
permutation both on the irreducible ordinary and Brauer
characters of $A$; in terms of columns of the decomposition
matrix this amounts to interchanging the first two columns.
But all ordinary and Brauer characters of $A'$ are real-valued.
Hence the Puig equivalence asserted by {\bf\ref{2HSandA8}}
does not commute with the self-equivalences
of the module categories of $A$ and $A'$ induced by
taking contragredient modules.

Actually, our proof of {\bf\ref{2HSandA8}} 
provides two distinct Puig equivalences,
one inducing the bijection between the simple $A$- and $A'$-modules
as indicated in the decomposition matrix above,
the other one inducing the bijection obtained by
interchanging its first two columns.

\medskip
(c)
As far as we have experienced,
it looks that most of all non-principal 
$3$-blocks with elementary abelian defect group $P$ of
order $9$ are just Morita equivalent to 
certain principal $3$-blocks with defect group $P$, see 
\cite{KoshitaniKunugiWaki2002,KoshitaniKunugiWaki2004,
KoshitaniKunugiWaki2008,KoshitaniMueller2010}, for instance.
One might be tempted to say that these non-principal blocks
are \emph{pseudo-principal}.
So, the non-principal block algebra $A$ considered here is 
even pseudo-principal in two ways, each leading to different 
`trivial' character; and the principal block algebra $A'$ 
is also pseudo-principal with a different `trivial' character.

However, there are non-principal $3$-blocks of
finite groups with defect group $P$ 
which are not pseudo-principal in the above sense,
that is they are not Morita equivalent to any principal $3$-block:
For example, it has been already noted in
\cite[Remark 4.4]{Koshitani2003} that the non-principal $3$-block 
with defect group $P$ of the Higman-Sims sporadic simple group {\sf HS} 
has this property, and the faithful $3$-blocks of $4.${\sf M}$_{22}$ 
described in \cite{MuellerSchaps2008} have as well.
\end{Remark}

\begin{Contents}
This paper is organised as follows:
In \S\ref{prel} we recall a few of the most important
ingredients of our proofs.
In \S\ref{***} we present the local data related to $G' = \mathfrak A_8$.
In \S\ref{**} we present the local data related to $G = 2.{\sf HS}$,
and relate the groups $G'$ and $G$; in particular we comment
on how the explicit embedding is achieved in a computational setting.
In \S\ref{*****} we proceed to give a stable equivalence for $A$
and its Brauer correspondent.
In \S\ref{******} we determine the images of the simple $A$-modules
with respect to this stable equivalence.
In \S\ref{proof} we finally complete the
proofs of {\bf\ref{MainTheorem}}, {\bf\ref{ADGCfor2HSforAllPrimes}} 
and {\bf\ref{2HSandA8}}, and we also give details on the 
phenomena in {\bf\ref{twoequivalences}}(a) and {\bf\ref{twoequivalences}}(b).

A further comment on the computational contents of the present
paper is in order:
As tools, we use the computer algebra system {\sf GAP} \cite{GAP},
to calculate with permutation groups as well as with ordinary
and Brauer characters. We also make use of the data library
\cite{CTblLib}, in particular allowing for easy access to the data
compiled in \cite{Atlas,ModularAtlas,ModularAtlasProject},
and of the interface \cite{AtlasRep} to the data library \cite{ModAtlasRep}.
Moreover, we use the computer algebra system {\sf MeatAxe} \cite{MA}
to handle matrix representations over finite fields,
as well as its extensions to compute submodule lattices 
\cite{LuxMueRin,MueLatt},
radical and socle series \cite{LuxWie},
and homomorphism spaces, endomorphism rings and direct sum decompositions 
\cite{LuxSzokeII,LuxSzoke}.
\end{Contents}

\begin{Notation/Definition}\label{NotationEquivalences}
(a)
Throughout this paper, we use the standard 
notation and terminology as is used in
\cite{Atlas,NagaoTsushima,Thevenaz}. We recall a few for convenience:

If $A$ and $B$ are finite dimensional $k$-algebras, where $k$ is a field,
we denote by $\mathrm{mod}{\text -}A$,
$A{\text -}\mathrm{mod}$ and
$A{\text -}\mathrm{mod}{\text -}B$
the categories of finitely generated right $A$-modules,
left $A$-modules and $(A,B)$-bimodules, respectively.
We write $M_A$, $_AM$ and $_AM_B$ when $M$ is 
a right $A$-module,
a left $A$-module and an $(A,B)$-bimodule. 
A module always refers to 
a finitely generated right module, unless stated otherwise.
We let $M^\vee =\Hom_A(M_A, A_A)$ be the $A$-dual
of the $A$-module $M$, so that $M^\vee$ becomes a left $A$-module 
via $(a\phi)(m) = a{\cdot}\phi(m)$ for $a \in A$, $\phi \in M^\vee$
and $m \in M$.
We denote by ${\mathrm{soc}}(M)$ and ${\mathrm{rad}}(M)$
the socle and the radical of $M$, respectively.
For simple $A$-modules $S_1, \cdots, S_n$, 
and positive integers $a_1, \cdots, a_n$,
we write that 
`\emph{$M = a_1 \times S_1 + \cdots + a_n \times S_n$,
as composition factors}'
when the set of all composition factors are
$a_1$ times $S_1$, $\cdots$, $a_n$ times $S_n$.
For another $A$-module $L$, we write $M{\Big |}L$ when
$M$ is isomorphic to a direct summand of $L$ as
an $A$-module.
If $A$ is self-injective,
the stable module category $\underline{\mathrm{mod}}{\text -}A$
is the quotient category of $\mathrm{mod}{\text -}A$
with respect to the projective $A$-homomorphisms, that is
those factoring through a projective module.

By $G$ we always denote a finite group, and we fix a
prime number $p$. Assume that $(\mathcal K, \mathcal O, k)$ is a
splitting $p$-modular system for all subgroups of $G$, that is
to say, $\mathcal O$ is a complete discrete valuation ring of
rank one such that its quotient field $\mathcal K$ is
of characteristic zero, and its residue field
$k=\mathcal O/\mathrm{rad}(\mathcal O)$ is of
characteristic $p$, and that $\mathcal K$ and $k$ are
splitting fields for all subgroups of $G$.
We denote by $k_G$ the trivial $kG$-module.
If $X$ is a $kG$-module, then we write $X^* =\mathrm{Hom}_k(X,k)$ 
for the {\sl contragredient} of $X$, 
namely, $X^*$ is again a $kG$-module
via $(\varphi g) (x) = \varphi(xg^{-1})$ for $x \in X$,
$\varphi \in X^*$ and $g \in G$;
if no confusion may arise we also call this the {\sl dual} of $X$.
Let $H$ be a subgroup of $G$, and let $M$ and $N$ be
a $kG$-module and an $kH$-module, respectively.
Then let ${M}{\downarrow}^G_H = {M}{\downarrow}_H$ be the
restriction of $M$ to $H$, and let 
${N}{\uparrow}_H^G = {N}{\uparrow}^G = 
(N \otimes_{kH}kG)_{kG}$ 
be the induction (induced module) of $N$ to $G$.

We denote by $\mathrm{Irr}(G)$ and $\mathrm{IBr}(G)$ the sets of
all irreducible ordinary and Brauer characters of $G$, respectively;
we write $1_G$ for the trivial character of $G$.
Since the character field 
$\Q(\chi):=\Q(\chi(g)\: ; \:g\in G)\subseteq\mathcal K$ 
of any character $\chi\in\mathrm{Irr}(G)$
is contained in a cyclotomic field, we may identify $\Q(\chi)$
with a subfield of the complex number field $\C$, hence we may think
of characters having values in $\C$. In particular,
we write $\chi^{*}$ for the complex conjugate of $\chi$,
where of course $\chi^{*}$ is the character of the 
$\mathcal K G$-module contragredient to the $\mathcal K G$-module
affording $\chi$.
If $A$ is a block algebra of $\mathcal OG$,
then we write $\mathrm{Irr}(A)$ and $\mathrm{IBr}(A)$ for
the sets of all characters in $\mathrm{Irr}(G)$ and $\mathrm{IBr}(G)$
which belong to $A$, respectively.

\medskip
(b)
Let $G'$ be another finite group, and let $V$ be an 
$(\mathcal OG, \mathcal OG')$-bimodule. 
Then we can regard $V$ as a right
$\mathcal O[G \times G']$-module
via $v{\cdot}(g,g') = g^{-1}vg'$ for $v \in V$ and $g, g' \in G$.
We denote by $\Delta G= \{ (g,g) \in G \times G \, | \, g \in G \}$.
Let $A$ and $A'$ be block algebras of $\mathcal OG$ and $\mathcal OG'$, 
respectively. Then we say that $A$ and $A'$ are {\sl Puig equivalent}
if $A$ and $A'$ have a common defect group $P$, 
and if there is a Morita equivalence between $A$ and $A'$
which is induced by an $(A,A')$-bimodule $\mathfrak M$ 
such that, as a right
$\mathcal O[G \times G']$-module, 
$\mathfrak M$ is a trivial source
module and $\Delta P$-projective. 
This is equivalent to a condition that
$A$ and $A'$ have source algebras which are isomorphic as
interior $P$-algebras, 
see \cite[Remark 7.5]{Puig1999} and \cite[Theorem 4.1]{Linckelmann2001}.
We say that $A$ and $A'$ are 
{\sl stably equivalent of Morita type} 
if there exists an $(A, A')$-bimodule 
$\mathfrak M$ such that both
${_A}\mathfrak M$ and $\mathfrak M_{A'}$ are
projective and that 
$_A(\mathfrak M \otimes_{A'} \mathfrak M^\vee)_A 
  \cong {_A}{A}{_A} \oplus
 ({\mathrm{proj}} \ (A,A){\text{-}}{\mathrm{bimod}})$
and
$_{A'} (\mathfrak M^\vee \otimes_A \mathfrak M)_{A'} 
\cong {_{A'}}{A'}{_{A'}} \oplus
 ({\mathrm{proj}} \ (A' ,A'){\text{-}}{\mathrm{bimod}})$.
We say that $A$ and $A'$ are 
{\sl splendidly stably equivalent of Morita type}
if $A$ and $A'$ have a common defect group $P$ and 
the stable 
equivalence of Morita type is induced by
an $(A,A')$-bimodule $\mathfrak M$ 
which is a trivial source
$\mathcal O[G \times G']$-module and is $\Delta P$-projective,
see \cite[Theorem~3.1]{Linckelmann2001}.

We say that $A$ and $A'$ are 
{\sl derived equivalent (or Rickard equivalent)} if
${\mathrm{D}}^b({\mathrm{mod}}{\text{-}}A)$ and
${\mathrm{D}}^b({\mathrm{mod}}{\text{-}}A')$ are
equivalent as triangulated categories,
where 
${\mathrm{D}}^b({\mathrm{mod}}{\text{-}}A)$
is the bounded derived category of ${\mathrm{mod}}{\text{-}}A$.
In that case, there even is a {\sl Rickard complex}
$M^{\bullet} \in
{\mathrm{C}}^b (A{\text{-}}{\mathrm{mod}}{\text{-}}A')$,
where the latter 
is the category of bounded complexes
of finitely generated $(A,A')$-bimodules,
all of whose terms are projective both as 
left $A$-modules and as right $A'$-modules, such that 
$M^{\bullet}\otimes_{A'}(M^{\bullet})^{\vee}\cong A$
in $K^b(A{\text{-}}{\mathrm{mod}}{\text{-}}A)$
and
$(M^{\bullet})^{\vee}\otimes_A M^{\bullet}\cong A'$
in $K^b(A'{\text{-}}{\mathrm{mod}}{\text{-}}A')$,
where $K^b(A{\text{-}}{\mathrm{mod}}{\text{-}}A)$ is the homotopy
category associated with 
${\mathrm{C}}^b (A{\text{-}}{\mathrm{mod}}{\text{-}}A)$;
in other words, in that case we even have 
$K^b({\mathrm{mod}}{\text{-}}A) \cong 
K^b({\mathrm{mod}}{\text{-}}A')$. 
We say that $A$ and $A'$ are 
{\sl splendidly derived equivalent} 
if $K^b({\mathrm{mod}}{\text{-}}A)$ and 
$K^b({\mathrm{mod}}{\text{-}}A')$ are equivalent
via a Rickard complex 
$M^{\bullet}\in {\mathrm{C}}^b (A{\text{-}}{\mathrm{mod}}{\text{-}}A')$
as above,
such that additionally each of its terms is a 
direct sum of $\Delta P$-projective 
trivial source modules as an
$\mathcal O[G \times G']$-module;
see \cite{Linckelmann1998,Linckelmann2001}.
\end{Notation/Definition}

\section{Preliminaries}\label{prel}

\begin{Lemma}[Scott]\label{Scott}
The following holds:
\begin{enumerate}
\renewcommand{\labelenumi}{\rm{(\roman{enumi})}}
    \item If $M$ is a trivial source $kG$-module,
then $M$ lifts uniquely (up to isomorphism) 
to a trivial source
$\mathcal O G$-lattice $\widehat M$.
    \item If $M$ and $N$ are both trivial source $kG$-modules,
then $[M,N]^G = (\chi_{\widehat M}, \chi_{\widehat N})^G$.
\end{enumerate}
\end{Lemma}

\begin{proof}
See \cite[II Theorem 12.4 and I Proposition 14.8]{Landrock} and 
\cite[Corollary 3.11.4]{Benson}.
\end{proof}

\begin{Lemma} [Linckelmann] 
\label{Linckelmann}
Let $A$ and $B$ be finite-dimensional $k$-algebras
such that $A$ and $B$ are both
self-injective and indecomposable as algebras, but not simple.
Suppose that there is an 
$(A,B)$-bimodule $M$ such that $M$ induces a stable equivalence
between the algebras $A$ and $B$.

\begin{enumerate}
  \renewcommand{\labelenumi}{\rm{(\roman{enumi})}}
    \item
If $M$ is indecomposable then for any simple $A$-module $S$, the $B$-module
$(S \otimes_A M)_B$ is non-projective and indecomposable.
    \item
If for all simple $A$-module $S$ the $B$-module 
$S \otimes_A M$ is simple, then $M$ induces a Morita
equivalence between $A$ and $B$.
\end{enumerate}
\end{Lemma}

\begin{proof}
(i) and (ii) respectively are given in 
\cite[Theorem 2.1(ii) and (iii)]{Linckelmann1996MathZ}.
\end{proof}

\begin{Lemma}[Fong-Reynolds]\label{FongReynolds}
Let $H$ be a normal subgroup of $G$, and let 
$A$ and $B$ be block algebras of $\mathcal OG$ and
$\mathcal OH$, respectively, such that $A$ covers $B$.
Let $T = T_G(B)$ be the inertial subgroup (stabiliser)
of $B$ in $G$. Then, there is a block algebra
$\tilde A$ of $\mathcal OT$ such that
$\tilde A$ covers $B$,
$1_A 1_{\tilde A} = 1_{\tilde A} 1_A = 1_{\tilde A}$, 
$A = \tilde A ^G$ (block induction), and the block
algebras $A$ and $\tilde A$ are Morita equivalent via a
pair 
$(1_A{\cdot}{\mathcal OG}{\cdot}1_{\tilde A}, \, 
  1_{\tilde A}{\cdot}{\mathcal OG}{\cdot}1_A)$, 
that is, the Morita equivalence is a Puig equivalence
and induces a bijection
$$
{\mathrm{Irr}}(\tilde A) \rightarrow
{\mathrm{Irr}}(A),   \quad
\tilde{\chi} \mapsto \tilde{\chi}{\uparrow}^G;
\qquad
{\mathrm{Irr}}(A) \rightarrow
{\mathrm{Irr}}(\tilde A),   \quad
\chi \mapsto {\chi}{\downarrow}_T{\cdot}1_{\tilde A}
$$
between $\mathrm{Irr}(\tilde A)$ and $\mathrm{Irr}(A)$,
and a bijection
$$
{\mathrm{IBr}}(\tilde A) \rightarrow
{\mathrm{IBr}}(A),   \quad
\tilde{\phi} \mapsto \tilde{\phi}{\uparrow}^G;
\qquad
{\mathrm{IBr}}(A) \rightarrow
{\mathrm{IBr}}(\tilde A),   \quad
\phi \mapsto {\phi}{\downarrow}_T{\cdot}1_{\tilde A}
$$
between $\mathrm{IBr}(\tilde A)$ and $\mathrm{IBr}(A)$,
\end{Lemma}

\begin{proof} See \cite[1.5.Theorem]{KoshitaniKunugiWaki2004} and
\cite[Chap.5, Theorem 5.10]{NagaoTsushima}.
\end{proof}

\begin{Remark}\label{FongReynoldsRemark}
In {\bf\ref{FongReynolds}} $\tilde A$ is called a
{\sl Fong-Reynolds correspondent} of $A$ and vice versa.
Note that there can be more than one Fong-Reynolds correspondent
in general.
\end{Remark}

\begin{Lemma}\label{StrippingOffMethod}
Let $A$ and $B$ be finite dimensional $k$-algebras
for a field $k$ such that $A$ and $B$ are both self-injective.
Let $F$ be a covariant functor such that
 \begin{enumerate}
  \renewcommand{\labelenumi}{\rm{(\arabic{enumi})}}
   \item $F$ is exact.
   \item If $X$ is a projective $A$-module, then $F(X)$ is a 
    projective $B$-module,
   \item $F$ induces a stable equivalence from
    $\mod{\text{-}}A$ to $\mod{\text{-}}B$.
 \end{enumerate}
Then the following holds:
\begin{enumerate}
 \renewcommand{\labelenumi}{\rm{(\roman{enumi})}}
 \item
{\sf (Stripping-off method, case of socle)} \, 
  Let $X$ be a projective-free $A$-module, and write
  $F(X) = Y \oplus ({\mathrm{proj}})$ for a projective-free $B$-module $Y$. 
  Let $S$ be a simple $A$-submodule of $X$, and set $T = F(S)$. 
  Now, if $T$ is a simple $B$-module, then we may assume that $Y$
  contains $T$ and that 
  $$ F(X/S)= Y/T \oplus ({\mathrm{proj}}) .$$
 \item 
{\sf (Stripping-off method, case of radical)} \,
  Similarly, 
  let $X$ be a projective-free $A$-module, and write
  $F(X) = Y \oplus ({\mathrm{proj}})$ for a projective-free $B$-module $Y$. 
  Let $X'$ be an $A$-submodule of $X$ such that $X/X'$ is simple,
  and set $T = F(X/X')$. Now, if $T$ is a simple $B$-module,
  then we may assume that $T$ is an epimorphic image of $Y$ and that
  $$ {\mathrm{Ker}}(F(X) \twoheadrightarrow T) ={\mathrm{Ker}}(Y
  \twoheadrightarrow T) \oplus ({\mathrm{proj}}) .$$
\end{enumerate}
\end{Lemma}

\begin{proof} See 
\cite[1.11.Lemma]{KoshitaniKunugiWaki2004} or
\cite[A.1.Lemma]{KoshitaniMuellerNoeske2011}.
\end{proof}

\begin{Lemma}\label{SourceVertex}
Let $H$ be a proper subgroup of $G$, and 
let $A$ and $B$ be block algebras of $kG$ and $kH$, respectively.
Now, let $M$ and $M'$ be finitely generated $(A,B)$-
and $(B,A)$-bimodules, respectively, which satisfy the following:
\begin{enumerate}
 \renewcommand{\labelenumi}{\rm{(\arabic{enumi})}}
   \item
${_A}M_B \mid 1_A{\cdot}kG{\cdot}1_B$ and
${_B}{M'}_A \mid 1_B{\cdot}kG{\cdot}1_A$. 
   \item
The pair $(M, M')$ induces a stable equivalence between
${\mathrm{mod}}{\text{-}}A$ and ${\mathrm{mod}}{\text{-}}B$.
\end{enumerate}
Then we get the following:
\begin{enumerate}
  \renewcommand{\labelenumi}{\rm{(\roman{enumi})}}
   \item Assume that $X$ is a non-projective indecomposable $kG$-module
    in $A$ with vertex $Q$. Then there exists a non-projective
    indecomposable $kH$-module $Y$ in $B$, unique up to isomorphism,
    such that $(X \otimes_A M)_B = Y \oplus ({\mathrm{proj}})$, and
    $Q^g$ is a vertex of $Y$ for some element $g \in G$ 
    {\rm{(}}and hence $Q^g \subseteq H${\rm{)}}. 
    Since $Q^g$ is also a vertex of $X$, this means that
    $X$ and $Y$ have at least one vertex in common.
   \item Assume that $Y$ is a non-projective indecomposable $kH$-module
    in $B$ with vertex $Q$. Then there exists a non-projective
    indecomposable $kG$-module $X$ in $A$, unique up to isomorphism,
    such that $(Y \otimes_B {M'})_A = X \oplus ({\mathrm{proj}})$, and
    $Q$ is a vertex of $X$.
   \item Let $X, Y$ and $Q \leqslant H$ be the as in {\rm{(i)}}.
    Then there is an indecomposable $kQ$-module $L$ such that $L$ is a
    source of both $X$ and $Y$. This means that $X$ and $Y$ have 
    at least one source in common.
   \item Let $X, Y$ and $Q \leqslant H$ be the same as in {\rm{(ii)}}.
    Then there is an indecomposable $kQ$-module $L$ such that $L$ is a
    source of both $X$ and $Y$. This means that $X$ and $Y$ have
    at least one source in common.
   \item Let $X, Y$, $Q$ and $L$ be the same as in {\rm{(iii)}}.
    In addition, suppose that $A$ and $B$ have a common defect group $P$
    {\rm{(}}and hence $P \subseteq H${\rm{)}} 
    and that $H \geqslant N_G(P)$. Let $f$
    be the Green correspondence with respect to $(G, P, H)$. If $ Q \in
    {\mathfrak A}(G, P, H)$, see \cite[Chap.4, \S 4]{NagaoTsushima}
    then we have $(X \otimes_A M)_B = f(X) \oplus ({\mathrm{proj}})$.
  \item Let $X$, $Y$, $Q$ and $L$ be the same as in {\rm{(ii)}}.
    Furthermore, as in {\rm{(v)}}, assume that $P$ is a common defect
    group of $A$ and $B$, and that $H \geqslant N_G(P)$, and let $f$
    and $\mathfrak A$ be the same as in {\rm{(v)}}. Now, if $Q \in
    {\mathfrak A}(G, P, H)$, then we have 
    $(Y \otimes_B M')_A = f^{-1}(Y) \oplus ({\mathrm{proj}})$.
\end{enumerate} 
\end{Lemma}

\begin{proof} See 
\cite[A.3.Lemma]{KoshitaniMuellerNoeske2011}.
\end{proof}

\begin{Lemma}\label{A5SemidirectC4}
Set $G = \mathfrak A_5 \rtimes C_4 = (\mathfrak A_5 \times 2).2$,
where the action on $C_4$ of $\mathfrak A_5$ is that
$C_4/C_2$ acts faithfully on $\mathfrak A_5$. 
Let $P$ be a Sylow $3$-subgroup of $\mathfrak A_5$
(and hence $P \cong C_3$).
\begin{enumerate}\renewcommand{\labelenumi}{\rm{(\roman{enumi})}}
    \item 
There is a faithful non-principal block algebra $A$ of $kG$
(that is, not having the central subgroup of order $2$ in its kernel) 
with defect group $P$.
    \item
We can write 
${\mathrm{Irr}}(A) = \{ \chi_1, \chi_2, \chi_3 \}$
such that
$\chi_1(1) = 1$, $\chi_2(1) = 4$, $\chi_3(1) = 5$
and 
$\chi_1(u) = \chi_2(u) = 1$
for any element $u \in P - \{1\}$.
Moreover, we can write
${\mathrm{IBr}}(A) = \{ \varphi_1, \varphi_2 \}$
such that the $3$-decomposition matrix of $A$ is
\begin{center}
{
{\rm
\begin{tabular}{l|cc}
            & $\varphi_1$ & $\varphi_2$   \\
\hline
$\chi_{1}$ & 1  &  . \\
$\chi_{2}$ & .  &  1 \\
$\chi_{3}$ & 1  &  1 \\
\end{tabular} 
}}
\end{center}
     \item
Set $H = N_G(P)$. Then
$H = \mathfrak S_3 \times C_4$,
and let $B$ be a block algebra
of $kH$ which is the Brauer correspondent of $A$.
      \item
Set $M = \mathfrak f(A)$, where $\mathfrak f$ is the Green
correspondence with respect to $(G \times G, \Delta P, G \times H)$.
Then, $M$ induces a Morita equivalence between 
$A$ and $B$ (and hence $M$ induces a Puig equivalence
between $A$ and $B$). Furthermore, the simple $kG$-modules in $A$
affording $\varphi_1$ and $\varphi_2$ 
are both trivial source $kG$-modules.
\end{enumerate}
\end{Lemma}

\begin{proof}
(i)-(iii) are easy.
(iv) follows from (i)-(iii) and 
\cite[Theorem 1.2]{KoshitaniKunugi2010}.
\end{proof}

\section{Green correspondences for $\mathfrak A_8$}\label{***}

\begin{Notation}\label{NotationA8}
We introduce some further notation which we use through out the rest of the
paper.
Let $G'$ be the alternating group on $8$ letters, namely,
$G' = \mathfrak A_8$. Since Sylow $3$-subgroups of $G'$ are 
isomorphic to $C_3 \times C_3$,
we can assume that $P$ is a Sylow $3$-subgroup of 
$\mathfrak A_8$ as well,
which is originally a defect group of $A$ and also a 
Sylow $3$-subgroup of $G = 2.${\sf HS},
see {\bf\ref{Notation2}} and {\bf\ref{Notation3}}.
There are exactly two conjugacy classes of $G'$ which contain
elements of order $3$, that is, $P$ has exactly two 
$G'$-conjugacy classes of subgroups of order $3$,
so we call them $Q$ and $R$, see
{\bf\ref{CenQ}}(ii), and see also \cite[p.22]{Atlas}.
Let $H' = N_{G'}(P)$, and hence $H' = P \rtimes D_8$;
note that the subgroups of order $3$ of $P$ still fall
into two $H'$-conjugacy classes. 
Let $A'$ and $B'$, respectively, be the principal block algebras
of $kG'$ and $kH'$. Thus $B' = kH'$.
\end{Notation}

\begin{Lemma}\label{CharactersOfA8}
\begin{enumerate} \renewcommand{\labelenumi}{\rm{(\roman{enumi})}}
    \item 
The $3$-decomposition matrix and the Cartan matrix
of $A'$, respectively, are the following:
{\rm
\begin{center}
\begin{tabular}{l|ccccc}
  & $k_{G'}$ & $7$ & $13$ & $28$ & $35$ \\
\hline
$\chi'_{1} $ & 1 & . & . & . & . \\
$\chi'_{7} $ & . & 1 & . & . & . \\
$\chi'_{14}$ & 1 & . & 1 & . & . \\ 
$\chi'_{20}$ & . & 1 & 1 & . & . \\ 
$\chi'_{28}$ & . & . & . & 1 & . \\ 
$\chi'_{35}$ & . & . & . & . & 1 \\ 
$\chi'_{56}$ & 1 & 1 & 1 & . & 1 \\ 
$\chi'_{64}$ & 1 & . & . & 1 & 1 \\ 
$\chi'_{70}$ & . & 1 & . & 1 & 1 \\ 
\end{tabular} 
\quad
\quad
\quad
\quad
\begin{tabular}{r|ccccc}
      & $P(k_{G'})$ & $P(7)$ & $P(13)$ & $P(28)$ & $P(35)$ \\
\hline
$k_{G'}$ & 4 & 1 & 2 & 1 & 2   \\ 
$7$      & 1 & 4 & 2 & 1 & 2   \\ 
$13$     & 2 & 2 & 3 & 0 & 1   \\ 
$28$     & 1 & 1 & 0 & 3 & 2   \\ 
$35$     & 2 & 2 & 1 & 2 & 4   \\ 
\end{tabular}
\end{center}
}
\smallskip
    \item 
All simple $kG'$-modules $k_{G'}, 7, 13, 28, 35$
in $A'$ have $P$ as their vertices.
\end{enumerate}
\end{Lemma}

\begin{proof}
(i) follows from \cite[$A_8$ (mod $3$)]{ModularAtlas}
and \cite[p.22]{Atlas},
and for (ii) see \cite[3.7.Corollary]{Knoerr}.
\end{proof}

\begin{Notation}\label{NotationB}
We use the notation
$\chi'_1, \cdots , \chi'_{70}$ and
$k_{G'}, 7, 13, 28, 35$ as in {\bf\ref{CharactersOfA8}},
where the numbers mean the degrees (dimensions) of characters
(modules).
\end{Notation}

\begin{Lemma}\label{CharacterTableOfBprime}
The following holds:
\begin{enumerate} \renewcommand{\labelenumi}{\rm{(\roman{enumi})}}
    \item 
The character table of 
$H' = P \rtimes D_8 \cong (C_3 \times C_3) \rtimes D_8$
is given as follows:
\smallskip

\begin{center}
\begin{tabular}{c|rrrrrrrrr}
{\rm{centraliser}} & $72$& $8$& $12$& $12$& $18$& $18$& $4$& $6$& $6$ \\
{\rm{element}}     & $1A$& $2A$& $2B$& $2C$& $3A$& $3B$& $4A$& $6A$& $6B$ \\
\hline
$\chi_{1a}$ & $1$& $1$& $1$& $1$& $1$& $1$& $1$& $1$& $1$ \\
$\chi_{1b}$ & $1$& $1$& $-1$& $-1$& $1$& $1$& $1$& $-1$& $-1$ \\
$\chi_{1c}$ & $1$& $1$& $-1$& $1$& $1$& $1$& $-1$& $-1$& $1$ \\
$\chi_{1d}$ & $1$& $1$& $1$& $-1$& $1$& $1$& $-1$& $1$& $-1$ \\
$\chi_2$ \  & $2$& $-2$& $0$& $0$& $2$& $2$& $0$& $0$& $0$ \\
$\chi_{4a}$ & $4$& $0$& $0$& $2$& $-2$& $1$& $0$& $0$& $-1$ \\
$\chi_{4b}$ & $4$& $0$& $0$& $-2$& $-2$& $1$& $0$& $0$& $1$ \\
$\chi_{4c}$ & $4$& $0$& $2$& $0$& $1$& $-2$& $0$& $-1$& $0$ \\
$\chi_{4d}$ & $4$& $0$& $-2$& $0$& $1$& $-2$& $0$& $1$& $0$ \\
\end{tabular}
\end{center}
Note that $\chi_{1b}$ is distinguished amongst the non-trivial
linear characters, for example 
by having an element of order $4$ in its kernel.

\smallskip
   \item 
$H'=\Inn(H')\vartriangleleft\Aut(H')$ such that $|\Aut(H')/H'|=2$,
and any non-inner automorphism of $H'$ 
induces a non-inner automorphism of $D_8=H'/P$, 
and interchanges the two conjugacy classes of subgroups of order $3$ of $P$.
In particular, there is an induced character table automorphism of 
$\Irr(H')$ interchanging 
$$\chi_{1c}\leftrightarrow\chi_{1d},\quad
  \chi_{4a}\leftrightarrow\chi_{4c},\quad
  \chi_{4b}\leftrightarrow\chi_{4d} .$$

\smallskip
   \item 
The $3$-decomposition matrix and the Cartan matrix
of $B' = kH' = k[P \rtimes D_8]$,
respectively, are the following:
{\rm
\begin{center}
\begin{tabular}{l|ccccc}
  & $1a$ & $1b$ & $1c$ & $1d$ & $2$ \\
\hline
$\chi_{1a}$ & 1  &  .    & .     &  .    &  . \\
$\chi_{1b}$ & .  &  1    & .     &  .    &  . \\
$\chi_{1c}$ & .  &  .    & 1     &  .    &  . \\ 
$\chi_{1d}$ & .  &  .    & .     &  1    &  . \\ 
$\chi_{2}$  & .  &  .    & .     &  .    &  1 \\ 
$\chi_{4a}$ & 1  &  .    & 1     &  .    &  1 \\ 
$\chi_{4b}$ & .  &  1    & .     &  1    &  1 \\ 
$\chi_{4c}$ & 1  &  .    & .     &  1    &  1 \\ 
$\chi_{4d}$ & .  &  1    & 1     &  .    &  1 \\ 
\end{tabular} 
\quad
\quad
\quad
\quad
\begin{tabular}{r|ccccc}
      & $P(1a)$ & $P(1b)$ & $P(1c)$ & $P(1d)$ & $P(2)$ \\
\hline
$1a$      & 3 & 0 & 1 & 1 & 2   \\ 
$1b$      & 0 & 3 & 1 & 1 & 2   \\ 
$1c$      & 1 & 1 & 3 & 0 & 2   \\ 
$1d$      & 1 & 1 & 0 & 3 & 2   \\ 
$2$       & 2 & 2 & 2 & 2 & 5   \\ 
\end{tabular}
\end{center}
}
\smallskip
    \item
All simple $kH'$-modules $1a, 1b, 1c, 1d, 2$
in $B'$ have $P$ as their vertices.
\end{enumerate}
\end{Lemma}

\begin{proof}
(i) and (ii) follow from an explicit computation with {\sf GAP} \cite{GAP},
the rest is easy.
\end{proof}

\begin{Notation}\label{NotationBprime}
We use the notation
$\chi_{1a}, \cdots , \chi_{4d}$ and
$1a = k_{H'}, 1b, 1c, 1d, 2$ as in {\bf\ref{CharacterTableOfBprime}},
where the numbers mean the degrees (dimensions) of characters
(modules).
\end{Notation}

\begin{Lemma}\label{TrivialSourceModulesInBprime}
The block algebra $B' = kH' = k[P \rtimes D_8]$ has exactly
$18$ non-isomorphic trivial source
modules over $k$. In fact, they are given in the
following list, where the diagrams are Loewy and socle series:
\smallskip
\begin{enumerate} \renewcommand{\labelenumi}{\rm{(\roman{enumi})}}
    \item 
Five PIM's $P(1a)$, $P(1b)$, $P(1c)$, $P(1d)$, $P(2)$:
\smallskip
{\rm
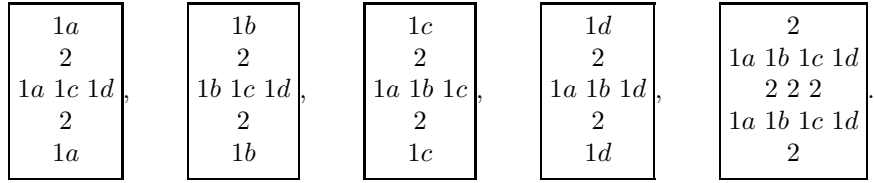
$$
\boxed{\begin{matrix}
1a \\
2 \\
1a \ 1c \ 1d \\
2 \\
1a
\end{matrix}},
\qquad
\boxed{\begin{matrix}
1b \\
2 \\
1b \ 1c \ 1d \\
2 \\
1b 
\end{matrix}},
\qquad
\boxed{\begin{matrix}
1c \\
2 \\
1a \ 1b \ 1c \\
2 \\
1c
\end{matrix}},
\qquad
\boxed{\begin{matrix}
1d \\
2 \\
1a \ 1b \ 1d \\
2 \\
1d
\end{matrix}},
\qquad
\boxed{\begin{matrix}
2 \\ 
1a \ 1b \ 1c \ 1d \\
2 \ 2 \ 2 \\
1a \ 1b \ 1c \ 1d \\
2
\end{matrix}}.
$$
}
\smallskip
    \item 
Five trivial source modules with vertex $P$: 
The simple modules $1a$, $1b$, $1c$, $1d$, $2$.
\smallskip
   \item
Eight trivial source modules with cyclic vertex of order $3$,
where we also give the associated trivial source characters,
see {\bf\ref{Scott}}:
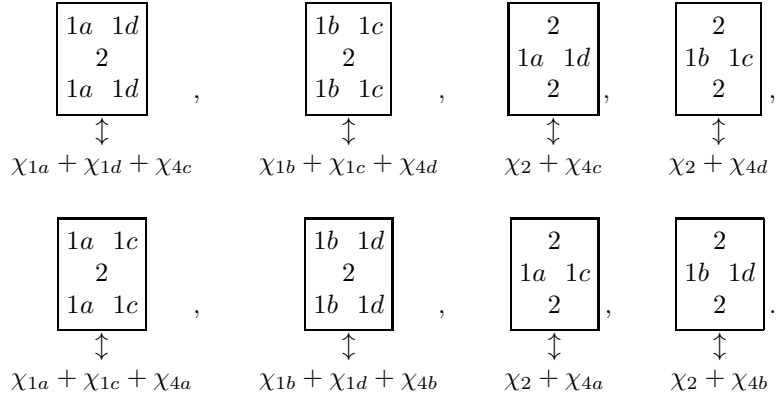
\smallskip
$$
\begin{matrix}
\boxed{
  \begin{matrix}
  1a \ \ 1d \\
     2      \\
  1a \ \ 1d 
  \end{matrix}
  }
            \\
  \updownarrow \\
\chi_{1a}+\chi_{1d}+\chi_{4c}
\end{matrix},
\qquad
\begin{matrix}
  \boxed{
  \begin{matrix}
  1b \ \ 1c \\
     2      \\
  1b \ \ 1c \\
  \end{matrix}
  }
  \\
  \updownarrow  \\ 
   
\chi_{1b}+\chi_{1c}+\chi_{4d}
\end{matrix},
\qquad
\begin{matrix}
  \boxed{
  \begin{matrix}
     2 \\
  1a \ \ 1d \\
     2     \\
  \end{matrix}
  }
   \\
  \updownarrow \\
\chi_2+\chi_{4c}
\end{matrix},
\qquad
\begin{matrix}
\boxed{
 \begin{matrix}
     2 \\
  1b \ \ 1c\\
     2  
 \end{matrix}
 }
            \\
  \updownarrow \\
\chi_2+\chi_{4d}
\end{matrix},
$$ 
\bigskip
$$
\begin{matrix}
   \boxed{
   \begin{matrix}
  1a \ \ 1c \\
     2      \\
  1a \ \ 1c \\
   \end{matrix}
   }  
  \\  
\updownarrow \\
\chi_{1a}+\chi_{1c}+\chi_{4a}
\end{matrix},
\qquad
\begin{matrix}
  \boxed{ 
  \begin{matrix}
  1b \ \ 1d \\
     2      \\
  1b \ \ 1d \\
   \end{matrix}
    }
    \\
  \updownarrow \\
\chi_{1b}+\chi_{1d}+\chi_{4b}
\end{matrix},
\qquad
\begin{matrix}
  \boxed{
  \begin{matrix}
     2 \\
  1a \ \ 1c \\
     2     \\
   \end{matrix}
   }
   \\
  \updownarrow \\
\chi_2+\chi_{4a}
\end{matrix},
\qquad
\begin{matrix}
   \boxed{
   \begin{matrix}
     2 \\
  1b \ \ 1d\\
     2     \\
   \end{matrix}
        }
     \\
  \updownarrow \\
\chi_2+\chi_{4b}
\end{matrix}.
$$
\end{enumerate}
\end{Lemma}

\begin{proof}
(i)
The structure of the PIM's is immediate as soon as we know that
${\mathrm{Ext}}_{kH'}^1(1x,1y)=0$ for all $x,y\in\{a,b,c,d\}$.
This in turn follows from the Ext-quiver of $B'$, which is given
as a quiver with relations in \cite[Section 4, Case 2]{Okuyama1997}.

To find the non-projective trivial source modules, we employ
\cite[Chap.4, Exc.10]{NagaoTsushima}. From that (ii) is immediate.
Moreover, this also yields the trivial source characters given in (iii),
from which it is easy to see, using the vanishing of 
${\mathrm{Ext}}_{kH'}^1(1x,1y)$ again, that the associated modules are
indecomposable.
\end{proof}

\begin{Lemma}\label{StableEquivalenceInA8}
An $(A', B')$-bimodule $\mathcal M'$ defined by
$\mathcal M' = f_{(G' \times G', \Delta P, G' \times H')}(A')$
induces a splendid stable equivalence of Morita type
between $A'$ and $B'$, namely by
$$ \mathcal F': {\mathrm{mod}}{\text{-}}A' \rightarrow
                {\mathrm{mod}}{\text{-}}B':
 X_{A'} \mapsto (X \otimes_{A'}{\mathcal M'})_{B'} .$$
In particular, $\mathcal F'$ fulfills the
assumptions of {\bf\ref{SourceVertex}}, and hence its assertions as well.
\end{Lemma}

\begin{proof}
This follows from 
\cite[Example 4.3]{Okuyama1997} and
\cite[Corollary 2]{Okuyama2000}.
\end{proof}

\begin{Lemma}\label{GreenCorrespondentA8}
Let $f'$ and $g'$ 
be the mutually inverse Green correspondences with respect to $(G', P, H')$.
Then the Green correspondents of simple modules are the following:
\smallskip
$$
\begin{array}{ccccc|ccccc}
g'(1a) &=   &\boxed{k_{G'}}  &\leftrightarrow &\chi'_1 
& & & f'(k_{G'}) &= &1a \\
& & & & & & & & \\
g'(1b)     &=   &\boxed{7} &\leftrightarrow &\chi'_{7} 
& & & f'(7) &= &1b \\
& & & & & & & & \\
g'(1c)     &=   
                 &
            \boxed{ \begin{matrix} 13  \\
                              k_{G'} \ \ 7 \\
                              13
                    \end{matrix}  }
                 &\leftrightarrow  &\chi'_{14} + \chi'_{20} 
& & & f'(13) &=  
&\boxed{\begin{matrix} 1c \ \\ \ 2 \ \\ 1c \ \end{matrix} } \\
& & & & & & & & \\
g'(1d)     &=   &\boxed{28}   &\leftrightarrow  &\chi'_{28} 
& & & f'(28) &= &1d
\\
& & & & & & & & \\
g'(2)      &=   &\boxed{35}  &\leftrightarrow  &\chi'_{35}\ 
& & & f'(35) &= &2
\end{array}
$$
On the left hand we also give the associated trivial source characters,
see {\bf\ref{Scott}}.
Recall that by considering $H'$ just as an abstract group
$\{1c,1d\}$ are 
indistinguishable, see {\bf\ref{CharacterTableOfBprime}},
but now note that by fixing $H'\leqslant G'$
and specifying $f'$, this defines $1c$ and $1d$ uniquely.
\end{Lemma}

\begin{proof}
This follows from \cite[Theorem]{Waki1989} and
\cite[Example 4.3]{Okuyama1997}.
\end{proof}

\section{$3$-Local structure for $2.${\sf HS}}\label{**}

\begin{Notation}\label{Notation2}
\mbox{}From now on, we assume that $G$ is the covering group
$2.${\sf HS} of the sporadic simple Higman-Sims group {\sf HS},
and hence $|G| = 2^{10}{\cdot}3^2{\cdot}5^3{\cdot}7{\cdot}11$,
see \cite[p.80]{Atlas}.
\end{Notation}

\begin{Lemma}\label{BrouesConjectureFor2HS}
We obtain the following:
\begin{enumerate}
  \renewcommand{\labelenumi}{\rm{(\roman{enumi})}}
      \item
In order to prove Brou{\'e}'s abelian defect group conjecture 
{\bf\ref{ADGC}} and Rickard's conjecture {\bf\ref{RickardConjecture}}
for $G = 2.${\sf HS}, it suffices to prove them for the case $p = 3$.
    \item
There exists a unique faithful
$3$-block $A$ with non-cyclic abelian 
defect group $P$, and $P$ is elementary abelian of order $9$,
and in order to prove 
Brou{\'e}'s abelian defect group conjecture for $G = 2.${\sf HS},
it suffices to prove it for this $3$-block $A$.
\end{enumerate}
\end{Lemma}

\begin{proof}
(i) Since the conjecture is proved when the defect group is
cyclic, we know from {\bf\ref{Notation2}} that it is enough to check
it for the primes $p \in \{ 2,3,5 \}$.
For $p = 2$ it follows from \cite[$HS$, (mod 2)]{ModularAtlas}
that there are a couple of blocks of $G$ and both have
non-abelian defect groups.
For $p = 5$ it follows from \cite[$HS$ (mod 5)]{ModularAtlas}
that there are a couple of blocks of $G$ which have noncyclic
defect groups and the defect group is non-abelian. 

(ii) Assume that $p = 3$. Then, again by
\cite[$HS$ (mod $3$)]{ModularAtlas}, there are three
3-blocks of $G$ which have noncyclic defect groups.
Those 3-blocks have defect groups which are elementary
abelian of order 9. Two of them are non-faithful and therefore
these two blocks show up in {\sf HS}.
For the  principal $3$-block of {\sf HS}, the conjectures
have been checked by Okuyama \cite[Example 4.8]{Okuyama1997}.
For the non-principal $3$-block of {\sf HS},
they have been verified in our previous paper 
\cite[0.2 Theorem(ii)]{KoshitaniKunugiWaki2002}.
Thus, the remaining untreated case is a unique faithful $3$-block $A$ 
of $G$ with noncyclic defect group.
\end{proof}

\begin{Notation}\label{Notation3}
\mbox{}From now on, we assume $p = 3$ and we use the notation
$A$ and $P$ as in {\bf\ref{BrouesConjectureFor2HS}}, that is,
$A$ is a block algebra of $kG$ with defect group
$P \cong C_3 \times C_3$.
Set $N = N_G(P)$, and let $A_N$ be a block algebra of $kN$
which is the Brauer correspondent of $A$.
Let $(P,e)$ be a maximal $A$-Brauer pair in $G$, namely,
$e$ is a block idempotent of $kC_G(P)$ such that 
${\mathrm{Br}}_P(1_A){\cdot}e = e$, see \cite{AlperinBroue},
\cite{BrouePuig1980} and \cite[\S 40]{Thevenaz}.
Set $H = N_G(P,e)$, namely,
$H = \{ g \in N_G(P) | g^{-1}e g = e \}$.
Let $B$ be a block algebra of $kH$ which is
a Fong-Reynolds correspondent of $A_N$, see {\bf\ref{FongReynolds}};
note that there are exactly two distinct
Fong-Reynolds correspondents of $A_N$, see {\bf\ref{LocalStructureII}}(iii).
\end{Notation}

\begin{Lemma}\label{DecompositionMatrixOf2HS}
\begin{enumerate}
  \renewcommand{\labelenumi}{\rm{(\roman{enumi})}}
    \item
The $3$-decomposition matrix of $A$
is given as follows:

\medskip
\begin{center}
\begin{tabular}{r|c|ccccc}
{\rm{degree}} & \cite[p.81]{Atlas} & $S_1$ & $S_2 = {S_1}^*$ 
                            & $S_3$ & $S_4$ & $S_5$ \\
\hline
 $176$ \hspace*{0.15em} & $\chi_{26}$ & $1$  &  .    & .     &  .    &  . \\
 $176^*$& $\chi_{27}$   & .    &  $1$  & .     &  .    &  .\\
 $616$ \hspace*{0.15em} & $\chi_{28}$ & $1$  &  .    & $1$   &  .    &  . \\
 $616^*$  & $\chi_{29}$ & .    &  $1$  & $1$   &  .    &  .\\
  $56$ \hspace*{0.15em} & $\chi_{25}$ & .    &  .    & .     & $1$   &  . \\
$1000$ \hspace*{0.15em} & $\chi_{32}$ & .    &  .    & .     &  .    &  $1$\\
$1792$ \hspace*{0.15em} & $\chi_{35}$ & $1$  &  $1$  & $1$   &  .    &  $1$\\
$1232$ \hspace*{0.15em} & $\chi_{33}$ & $1$  &  .    & .     & $1$   &  $1$\\
$1232^*$  & $\chi_{34}$ & .    &  $1$  & .     & $1$   &  $1$\\
\end{tabular} 
\end{center}

\bigskip\noindent
where $S_1, \cdots, S_5$ are non-isomorphic simple $kG$-modules in $A$
whose $k$-dimensions are $176$, $176$, $440$, $56$, $1000$, respectively.
The simples $S_1$ and $S_2$ are dual to each other, while the
remaining are self-dual. There are three pairs
$(\chi_{26}, \chi_{27})$, $(\chi_{28}, \chi_{29})$ and
$(\chi_{33}, \chi_{34})$ of complex conjugate characters,
and all the other $\chi_i$'s are real-valued.
\medskip
\item
All simple kG-modules $S_1, \cdots , S_5$ in $A$ have
$P$ as a vertex.
\end{enumerate}
\end{Lemma}

\begin{proof}
(i) follows from
\cite[$HS$ (mod $3$)]{ModularAtlas}
and \cite[p.81]{Atlas},
and for (ii) see \cite[3.7.Corollary]{Knoerr}.
\end{proof}

\begin{Notation}\label{chi-Si}
We use the notation 
$\chi_{26}, \chi_{27}, \chi_{28}, \chi_{29}, 
 \chi_{25}, \chi_{32}, \chi_{35}, \chi_{33}, \chi_{34}$,
and $S_1, \cdots , S_5$ as in
{\bf\ref{DecompositionMatrixOf2HS}}.
\end{Notation}

\begin{Lemma}\label{LocalStructureI}
The following holds:
\begin{enumerate}
  \renewcommand{\labelenumi}{\rm{(\roman{enumi})}}
    \item
$N = N_G(P) = 2.\Big(2 \times (P \rtimes SD_{16})\Big)
   = P \rtimes L$ for a subgroup $L$ of $N$ with
$|L| = 2^6$ such that $L \vartriangleright Z \cong C_4$ and
$L/Z \cong SD_{16}$. Moreover, $L/Z$ acts non-trivially on $Z$,
with kernel isomorphic to $D_8$. (Recall that $SD_{16}$ has
a unique subgroup isomorphic to $D_8$.)
    \item
$C_G(P) = Z \times P$ and $L/Z$ acts faithfully on $P$. 
(Note that $SD_{16}$ is a Sylow $2$-subgroup of $\GL_2(3)$.)
    \item 
$Z = O_{3'}(C_G(P)) = O_{3'}(N)$, and we can write
${\mathrm{Irr}}(Z) = \{\psi_0, \psi_1, \psi_2, \psi_3\}$ 
such that
$\psi_i(z) = {\sqrt{-1}\,^i}$ for $i = 0, 1, 2, 3$,
where $z$ is a generator of $Z \cong C_4$,
and $\sqrt{-1}\in \mathcal O$ is a fixed $4$-th root of unity.
Moreover, we have
$T_N(\psi_i) = G$ for $i = 0, 2$, while for $j = 1,3$ we have
$$ T_N(\psi_1) = T_N(\psi_3) \lneqq G \text{ such that }
T_N(\psi_j)/C_G(P) \cong D_8 .$$
    \item
$kN=A_0 \oplus A_2 \oplus A_N$ as block algebras,
having inertial quotients $SD_{16}$, $SD_{16}$ and $D_8$, respectively.
Here, $A_0$ is the principal block algebra of $kN$, covering $\psi_0$,
while $A_2$ covers $\psi_2$; hence $A_N$ is the faithful block algebra
being the Brauer correspondent of $A$.
Moreover, $A_0 \cong A_2 \cong k[P \rtimes SD_{16}]$ as $k$-algebras,
and 
$$ A_N \cong {\mathrm{Mat}}_2(k[P \rtimes D_8]) $$ 
as $k$-algebras,
where $A_N$ has $k[P \rtimes D_8]$
as its source algebra.  
    \item
The $3$-decomposition matrix of $A_N$ is given as follows:
{\rm
\begin{center}
\begin{tabular}{l|ccccc}
  & $2\alpha$ & $2\beta$ & $2\gamma$ & $2\delta$ & $4$ \\
\hline
$\chi_{2\alpha}$ & 1  &  .    & .     &  .    &  . \\
$\chi_{2\beta}=\chi_{2\alpha}^\ast $ 
                 & .  &  1    & .     &  .    &  . \\
$\chi_{2\gamma}$ & .  &  .    & 1     &  .    &  . \\
$\chi_{2\delta}$ & .  &  .    & .     &  1    &  . \\
$\chi_{4}$       & .  &  .    & .     &  .    &  1 \\
$\chi_{8\alpha}$ & 1  &  .    & 1     &  .    &  1 \\
$\chi_{8\beta}$  & .  &  1    & .     &  1    &  1 \\
$\chi_{8\gamma}=\chi_{8\beta}^\ast$
                 & 1  &  .    & .     &  1    &  1 \\
$\chi_{8\delta}=\chi_{8\alpha}^\ast$ 
                 & .  &  1    & 1     &  .    &  1 \\
\end{tabular}
\end{center}
}
\bigskip\noindent
where the numbers mean the degrees (dimensions) of characters (modules).
Note that $2\beta = {2\alpha}^*$ and that $2\gamma$, $2\delta$ and $4$
are all self-dual, but apart from this the characters of degree $2$
are indistinguishable: Apart from the character table automorphism 
of $\Irr(A_N)$ induced by complex conjugation there is another 
one interchanging
$$ \chi_{2\gamma}\leftrightarrow\chi_{2\delta},\quad
   \chi_{8\alpha}\leftrightarrow\chi_{8\gamma},\quad
   \chi_{8\beta}\leftrightarrow\chi_{8\delta} .$$
\end{enumerate}
\end{Lemma}

\begin{proof}
(i)--(ii) follow from explicit computation with {\sf GAP} \cite{GAP},
and (iii) is an immediate consequence.

(iv)--(v) It follows from (iii) and \cite[Theorem 2]{Morita1951} that
$kG = A_0 \oplus A_2 \oplus A_N$ where
$$
A_0 \cong 
 {\mathrm{Mat}}_{|G:T_G(\psi_0)| \psi_0(1)}
 \Big( k^{\alpha}[T_G(\psi_0)/Z] \Big)
 \cong  k^{\alpha} [P \rtimes SD_{16}],
$$ 
$$
A_1 \cong 
 {\mathrm{Mat}}_{|G:T_G(\psi_2)| \psi_2(1)}
 \Big( k^{\beta}[T_G(\psi_2)/Z] \Big)
 \cong k^{\beta} [P \rtimes SD_{16}],
$$ 
$$
A_N \cong {\mathrm{Mat}}_{|G: T_G(\psi_1)| \psi_1(1)}
      \Big( k^{\gamma}[T_G(\psi_1)/Z] \Big)
      \cong {\mathrm{Mat}}_2 \Big( k^{\gamma} [P \rtimes D_8] \Big),
$$
as $k$-algebras, for some  
$\alpha, \beta \in {\mathrm{H}}^2(SD_{16}, k^\times)$,
and $\gamma \in {\mathrm{H}}^2(D_8, k^\times)$.
Since ${\mathrm{H}}^2(SD_{16}, k^\times) = 0$ 
by \cite[Proof of Corollary (2J)]{Kiyota1984}, we can
assume $\alpha = \beta = 1$. 

On the other hand, 
$|{\mathrm{H}}^2(D_8, k^\times)| = 2$ by
\cite[V Satz 25.6]{Huppert1966}.
But now the assertion in (v) follows
by explicit computation with {\sf GAP} \cite{GAP},
in particular we get that there are $9$ irreducible 
ordinary characters belonging to $A_N$, and $5$ 
irreducible Brauer characters.
Hence, by \cite[Page 34 Table 1]{Kiyota1984} we infer $\gamma = 1$. 
Finally, the statement about source algebras follows from
\cite[Proposition 14.6]{Puig1988JAlg},
see \cite[(45.12)Theorem]{Thevenaz} and
\cite[Theorem 13]{AlperinLinckelmannRouquier}. 

Note that the decomposition matrix of $A_N$ given above coincides with that
of $B'$ in {\bf\ref{CharacterTableOfBprime}}. This will of course
turn out to be no accident, but by the current state of knowledge
we cannot avoid the explicit computation to proceed as above.
\end{proof}

\begin{Notation}\label{QandRandZ}
We use the notation $L$ and $Z$, 
as in {\bf\ref{LocalStructureI}}.
Moreover, let $z$ be a generator of $Z \cong C_4$.
We also use the notation 
$\chi_{2\alpha}, \chi_{2\beta}, \chi_{2\gamma}, \chi_{2\delta},
\chi_{4}, \chi_{8\alpha}, \chi_{8\beta}, \chi_{8\gamma}, 
\chi_{8\delta}$
and $2\alpha, 2\beta, 2\gamma, 2\delta, 4$ as in 
\textbf{\ref{LocalStructureI}}.
\end{Notation}

\begin{Lemma}\label{TrivialSourceModulesInA_N}
The block algebra $A_N$ has exactly
$18$ non-isomorphic trivial source
modules over $k$. In fact, they are given in the
following list, in which the diagrams are Loewy and socle
series and we use the same notation as in {\bf\ref{QandRandZ}}.
\smallskip
\renewcommand{\labelenumi}{\rm{(\roman{enumi})}}\begin{enumerate}
    \item 
Five PIM's: $P(2\alpha)$, $P(2\beta)$, $P(2\gamma)$, $P(2\delta)$, $P(4)$.
\smallskip
{\rm
$$
\boxed{\begin{matrix}
2\alpha \\
4 \\
2\alpha \ 2\gamma \ 2\delta \\
4 \\
2\alpha
\end{matrix}},
\qquad
\boxed{\begin{matrix}
2\beta \\
4 \\
2\beta \ 2\gamma \ 2\delta \\
4 \\
2\beta 
\end{matrix}},
\qquad
\boxed{\begin{matrix}
2\gamma \\
4 \\
2\alpha \ 2\beta \ 2\gamma \\
4 \\
2\gamma
\end{matrix}},
\qquad
\boxed{\begin{matrix}
2\delta \\
4 \\
2\alpha \ 2\beta \ 2\delta \\
4 \\
2\delta
\end{matrix}},
\qquad
\boxed{\begin{matrix}
4 \\ 
2\alpha \ 2\beta \ 2\gamma \ 2\delta \\
4 \ 4 \ 4 \\
2\alpha \ 2\beta \ 2\gamma \ 2\delta \\
4
\end{matrix}}.
$$
}
\smallskip
    \item 
Five trivial source modules with a vertex $P$: 
$2\alpha$, $2\beta$, $2\gamma$, $2\delta$, $4$.
\smallskip
   \item
Eight trivial source modules with cyclic vertex of order $3$,
where we also give the associated trivial source characters,
see {\bf\ref{Scott}}:
\smallskip
$$
\begin{matrix}
\boxed{
  \begin{matrix}
  2\alpha \ \ 2\delta \\
     4      \\
  2\alpha \ \ 2\delta
  \end{matrix}
  }
            \\
  \updownarrow \\
\chi_{2\alpha}+\chi_{2\delta}+\chi_{8\gamma}
\end{matrix},
\qquad
\begin{matrix}
  \boxed{
  \begin{matrix}
  2\beta \ \ 2\gamma \\
     4      \\
  2\beta \ \ 2\gamma \\
  \end{matrix}
  }
  \\
  \updownarrow  \\ 
\chi_{2\beta}+\chi_{2\gamma}+\chi_{8\delta}
\end{matrix},
\qquad
\begin{matrix}
\boxed{
 \begin{matrix}
     4 \\
  2\alpha \ \ 2\delta\\
     4  
 \end{matrix}
 }
            \\
  \updownarrow \\
\chi_4+\chi_{8\gamma}
\end{matrix},
\qquad
\begin{matrix}
  \boxed{
  \begin{matrix}
     4 \\
  2\beta \ \ 2\gamma \\
     4     \\
  \end{matrix}
  }
   \\
  \updownarrow \\
\chi_4+\chi_{8\delta}
\end{matrix},
$$
\bigskip
$$
\begin{matrix}
  \boxed{
   \begin{matrix}
  2\alpha \ \ 2\gamma \\
     4      \\
  2\alpha \ \ 2\gamma \\
   \end{matrix}
   }  
  \\  
\updownarrow \\
\chi_{2\alpha}+\chi_{2\gamma}+\chi_{8\alpha}
\end{matrix},
\qquad
\begin{matrix}
  \boxed{ 
  \begin{matrix}
  2\beta \ \ 2\delta \\
    4      \\
  2\beta \ \ 2\delta \\
   \end{matrix}
    }
    \\
  \updownarrow \\
\chi_{2\beta}+\chi_{2\delta}+\chi_{8\beta}
\end{matrix},
\qquad
\begin{matrix}
  \boxed{
  \begin{matrix}
     4 \\
  2\alpha \ \ 2\gamma \\
     4     \\
   \end{matrix}
   }
   \\
  \updownarrow \\
\chi_4+\chi_{8\alpha}
\end{matrix},
\qquad
\begin{matrix}
  \boxed{
   \begin{matrix}
     4 \\
  2\beta \ \ 2\delta \\
     4     \\
   \end{matrix}
        }
     \\
  \updownarrow \\
\chi_4+\chi_{8\beta}
\end{matrix} .
$$
\end{enumerate}
\end{Lemma}

\begin{proof}
This follows from {\bf\ref{TrivialSourceModulesInBprime}} 
and {\bf\ref{LocalStructureI}}(iv).
\end{proof}

\begin{Lemma}\label{LocalStructureII}
The following holds:
\begin{enumerate}
  \renewcommand{\labelenumi}{\rm{(\roman{enumi})}}
    \item
$H = Z \times (P \rtimes D_8) = C_G(P) \rtimes D_8$, hence
$H \vartriangleleft N$ such that $|N/H| = 2$.
      \item
We have the block decomposition
$$ kH \ = \ B_0 \oplus B_1 \oplus B_2 \oplus B_3 ,$$
where the block $B_i$ covers $\psi_i\in\Irr(Z)$, for $i=0,1,2,3$.
      \item
Hence both $B_1$ and $B_3$ are Fong-Reynolds correspondents of $A_N$,
see {\bf\ref{Notation3}}.
      \item
We can write 
${\mathrm{IBr}}(B_i) = 
\{ 1\alpha_i, 1\beta_i, 1\gamma_i, 1\delta_i, 2_{B_i}\}$,
for $i = 0, 1, 2, 3$,
so that we have
$$ 2x{\downarrow}^N_H = 1x_1 \oplus 1x_3, \text{ for each }
x \in \{ \alpha, \beta, \gamma, \delta \}, \text{ and }
4{\downarrow}^N_H = 2_{B_1} \oplus 2_{B_3} .$$
\end{enumerate}
\end{Lemma}

\begin{proof}
This follows from {\bf\ref{LocalStructureI}} 
and {\bf\ref{FongReynolds}}.
\end{proof}

\begin{Notation}\label{Notation6}
We use the notation
$1\alpha_i, 1\beta_i, 1\gamma_i, 1\delta_i, 2_{B_i}$
as in {\bf\ref{LocalStructureII}}.
\end{Notation}

\begin{Lemma}\label{2HScontainsA8}
The following holds:
\begin{enumerate}
  \renewcommand{\labelenumi}{\rm{(\roman{enumi})}}
\item
The group $G = 2.{\sf HS}$ has a unique conjugacy class of
subgroups isomorphic to $G' = \mathfrak A_8$.

\item
Fixing an embedding of $G'$ into $G$, and a Sylow $3$-subgroup
$P$ of $G'$, we have the configuration of groups as
depicted in Table {\bf\ref{2HScontainsA8table}},
where the numbers between two boxes are indices 
between the two corresponding groups.
\end{enumerate}
\end{Lemma}

\begin{table}
\caption{$G' = \mathfrak A_8$ as a subgroup of $G = 2.{\sf HS}$.}
\label{2HScontainsA8table}
\begin{center}
\unitlength=1pt
\boxed{
\begin{picture}(360,275)(-100,-260)
\put(50,0){\framebox(60,12){$G = 2.${\sf HS}}}
\put(-100,-100){\framebox(160,12){$N = N_G(P) 
                                     = 2.(2\times (P \rtimes SD_{16}))$}}
\put(-100,-150){\framebox(160,12){$H = N_G(P,e) 
                                     = Z\times (P \rtimes D_8)$}}
\put(200,-50){\framebox(60,12){$G' = \mathfrak A_8$}}
\put(35,-200){\framebox(120,12){$H' = N_{G'}(P)= P \rtimes D_8$}}
\put(-100,-225){\framebox(120,12){$C_G(P)=Z\times P$}}
\put(35,-256){\framebox(120,12){$P=C_{G'}(P)$}}
\put(-20,-150){\line(0,-1){63}}
\put(20,-225){\line(4,-1){75}}

\put(80,0){\line(4,-1){149}}
\put(80,0){\line(-1,-1){87}}
\put(-20,-100){\line(0,-1){38}}
\put(-20,-150){\line(3,-1){113}}
\put(230,-50){\line(-1,-1){137}}
\put(95,-200){\line(0,-1){44}}
\put(-18,-187){\makebox(5,10){$8$}}
\put(50,-233){\makebox(5,10){$4$}}

\put(-18,-124){\makebox(5,10){$2$}}
\put(50,-174){\makebox(5,10){$4$}}
\put(97,-226){\makebox(5,10){$8$}}
\end{picture}
}
\end{center}

\bigskip
\noindent\hrule
\end{table}

\begin{proof}
This follows from \cite[pp.80--81]{Atlas},
{\bf\ref{NotationA8}}, {\bf\ref{LocalStructureI}} and
{\bf\ref{LocalStructureII}}.
\end{proof}

\begin{Remark}\label{Computations}
In view of the group theoretic configuration given in
{\bf\ref{2HScontainsA8}}, a few more detailed comments
on how computations in {\sf GAP} \cite{GAP} are actually done 
are in order:

The starting point is the smallest faithful permutation 
representation of $G$ on $704$ points, available in terms of 
standard generators, see \cite{Wilson}, in \cite{ModAtlasRep};
we choose this realisation of $G$ once and for all.
Moreover, there is a maximal subgroup of $G$ isoclinic to
$\mathfrak S_8\times 2$, a generating set of which in terms
of the standard generators of $G$ is available in \cite{ModAtlasRep} 
as well. Hence going over to the derived subgroup of the latter,
we find a subgroup $G'\leqslant G$ isomorphic to $\mathfrak A_8$; we keep
$G'$ fixed all the time.
Finally, we compute a Sylow $3$-subgroup $P$ of $G'$,
and keep this fixed as well.
Since the other groups appearing in the diagram in
{\bf\ref{2HScontainsA8}} are uniquely determined from
this by group theoretic properties, we have thus
achieved a concrete realisation of the above configuration of groups.

Using this setting we have all kinds of computational tools
at our disposal: In particular, we are able to compute the 
ordinary and Brauer character tables of the groups
$N$, $H$, and $H'$ explicitly, as well as the restriction 
or induction of characters between these groups.
Moreover, we may fetch representations of $G$ and $G'$
from \cite{ModAtlasRep}, or compute them using the 
{\sf MeatAxe} \cite{MA}, and restrict them explicitly 
to $N$, $H$, and $H'$, respectively, in order to analyse
the restrictions with the {\sf MeatAxe} \cite{MA} and its
extensions. 

Note that the explicit results in {\bf\ref{CharacterTableOfBprime}}
and {\bf\ref{LocalStructureI}} have been obtained in that
setting already. Moreover, by restricting the representation $28$, 
see {\bf\ref{NotationB}}, from $G'$ to $H'$ we are able to
compute its Green correspondent $f'(28)$, see
{\bf\ref{GreenCorrespondentA8}}, and thus to identify 
the representation $1d$, see {\bf\ref{NotationBprime}}.
In the same spirit we obtain the following,
where we recall that so far, we are not able to tell
$2\gamma$ and $2\delta$ apart, see {\bf\ref{LocalStructureI}}(v):
\end{Remark}

\begin{Lemma}\label{StructureOfOH2}
The following holds:
\begin{enumerate}
  \renewcommand{\labelenumi}{\rm{(\roman{enumi})}}
      \item
For $i=1,3$ the restriction functor
${\mathrm{Res}}{\downarrow}^H_{H'}$ induced 
by the $(B_i,B')$-bimodule ${}_{B_i}(B_i)_{B'}$ 
induces a Puig equivalence
${\mathrm{mod}}{\text{-}}B_i \rightarrow {\mathrm{mod}}{\text{-}}B'$.
      \item
$\{1\alpha_1{\downarrow}^H_{H'},1\alpha_3{\downarrow}^H_{H'}\}=\{1a,1b\}
=\{1\beta_1{\downarrow}^H_{H'},1\beta_3{\downarrow}^H_{H'}\}$,
where $1\alpha_i{\downarrow}^H_{H'}\neq 1\beta_i{\downarrow}^H_{H'}$;
hence 
$$ 2\alpha{\downarrow}^N_{H'} = 
   (2\alpha^*){\downarrow}^N_{H'} = 
   2\beta{\downarrow}^N_{H'} 
= 1a \oplus 1b .$$
      \item
$1\gamma := 1\gamma_1{\downarrow}^H_{H'} = 1\gamma_3{\downarrow}^H_{H'}$
and
$1\delta := 1\delta_1{\downarrow}^H_{H'} = 1\delta_3{\downarrow}^H_{H'}$,
where 
$\{1\gamma,1\delta\}=\{1c,1d\}$; 
hence 
$$ 2\gamma{\downarrow}^N_{H'} = 1\gamma \oplus 1\gamma
\text{ and }
2\delta{\downarrow}^N_{H'} = 1\delta \oplus 1\delta .$$
      \item
$2_{B_1}{\downarrow}^H_{H'} = 2_{B_3}{\downarrow}^H_{H'} = 2$; hence
$4{\downarrow}^N_{H'} = 2 \oplus 2$.
\end{enumerate}
\end{Lemma}

\begin{proof}
Most of the assertions follow from {\bf\ref{LocalStructureII}},
while the identification of the explicit restrictions 
$1x_i{\downarrow}^H_{H'}$,
for $x\in\{\alpha,\beta,\gamma,\delta\}$,
follows from explicit computations in {\sf GAP} \cite{GAP}.
\end{proof}

\section{Stable equivalences for $2.${\sf HS} }\label{*****}

\begin{Notation}\label{Notation4}
Recall, first of all, the notation
$G$, $A$, $P$, $N$, $H$, $B$, $e$  
as in {\bf\ref{Notation2}}, {\bf\ref{Notation3}}
and {\bf\ref{QandRandZ}}.
Let $i$ and $j$ respectively be source idempotents of
$A$ and $B$ with respect to $P$.
As remarked in \cite[pp.821--822]{Linckelmann2001},
we can take $i$ and $j$ such that
${\mathrm{Br}}_P(i){\cdot}e = 
 {\mathrm{Br}}_P(i) \not= 0$ 
and that
${\mathrm{Br}}_P(j){\cdot}e = 
 {\mathrm{Br}}_P(j) \not= 0$.
Set $G_P = C_G(P) = C_H(P) = H_P$.

Moreover, letting $Q\leqslant P$ be a subgroup of order $3$,
we set $G_Q = C_G(Q)$ and $H_Q = C_H(Q)$. 
By replacing $e_Q$ and $f_Q$ (if necessary), we may assume that
$e_Q$ and $f_Q$ respectively are block idempotents of $kG_Q$ and
$kH_Q$ such that $e_Q$ and $f_Q$ are determined 
by $i$ and $j$, respectively.
Namely, 
${\mathrm Br}_Q(i){\cdot}e_Q =
 {\mathrm Br}_Q(i)$
and
${\mathrm Br}_Q(j){\cdot}f_Q =
 {\mathrm Br}_Q(j)$.
Let $A_Q = kG_Q{\cdot}e_Q$ and 
$B_Q = kH_Q{\cdot}f_Q$, so that
$e_Q = 1_{A_Q}$ and $f_Q = 1_{B_Q}$.
\end{Notation}

\begin{Lemma}\label{CenQ}
The following holds:
\begin{enumerate} \renewcommand{\labelenumi}{\rm{(\roman{enumi})}}
    \item
All elements in $P - \{ 1 \}$ are conjugate in $N$,
and hence in $G$; actually $P - \{1\} \subseteq 3A$,
where $3A$ is a conjugacy class
of $G$ following the notation in \cite[pp.80--81]{Atlas}.
Thus all subgroups of $P$ of order $3$ are conjugate in $N$,
and hence in $G$.
    \item
The elements in $P - \{ 1 \}$ fall into two conjugacy classes of $H$.
Thus $P$ has exactly two $H$-conjugacy classes of subgroups 
of order $3$; we call them $Q$ and $R$.
Note that we can use the same notation $Q$ and $R$ 
as in {\bf\ref{NotationA8}}.
    \item
$H_Q = C_N(Q) = Z \times Q \times {\mathfrak S_3}$
and $H_R = C_N(R) = Z \times R \times {\mathfrak S_3}$,
so that 
$H_Q/Q \cong C_4 \times {\mathfrak S_3}$
and
$H_R/R \cong C_4 \times {\mathfrak S_3}$.
    \item
$G_Q = Q \times (\mathfrak A_5 \rtimes Z)
    \cong Q \times (\mathfrak A_5 \times 2).2$,
so that
$G_Q/Q 
          \cong (\mathfrak A_5 \times 2).2$.
\end{enumerate}
\end{Lemma}

\begin{proof}
(i)--(iii) follow from {\bf\ref{LocalStructureI}} and 
{\bf\ref{LocalStructureII}},
(iv) follows from an explicit computation in 
{\sf GAP} \cite{GAP}.
\end{proof}

\begin{Lemma}\label{LocalPuig}
Let $\mathcal M_Q$ be the unique (up to isomorphism)
indecomposable direct summand of
$A_Q{\downarrow}^{G_Q \times G_Q}_{G_Q \times H_Q}{\cdot}1_{B_Q}$
with vertex $\Delta P$.
Then, a pair $(\mathcal M_Q, \mathcal M_Q^{\vee})$ induces
a Puig equivalence 
between $A_Q$ and $B_Q$.
\end{Lemma}

\begin{proof}
Note first that $\mathcal M_Q$ exists by
\cite[2.4.Lemma]{KoshitaniKunugiWaki2008}.
Then using {\bf\ref{CenQ}}(iii) and (iv), as well as
{\bf\ref{A5SemidirectC4}}(iv),
the assertion follows as in 
\cite[Proof of 6.2.Lemma]{KoshitaniKunugiWaki2008},
by going over to the central quotients $G_Q/Q$ and $H_Q/Q$
and their blocks dominating $A_Q$ and $B_Q$, respectively, 
and applying \cite[1.2.Theorem]{KoshitaniKunugi2010}
and \cite[Theorem]{KoshitaniKunugi2005}.
\end{proof}

\begin{Lemma}\label{StableEquivalenceFor2HS} 
The $(A, B)$-bimodule $1_A{\cdot}kG{\cdot}1_B$ has a
unique (up to isomorphism) indecomposable direct summand
${\mathcal M}$ with vertex $\Delta P$.
Moreover, the functor
$$ \mathcal F: {\mathrm{mod}}{\text{-}}A \rightarrow
   {\mathrm{mod}}{\text{-}}B: X_A \mapsto (X \otimes_A{\mathcal M})_B $$
induces a splendid stable equivalence of Morita type between
$A$ and $B$. 
In particular, $\mathcal F$ fulfils the 
assumptions of {\bf\ref{SourceVertex}}, and hence its assertions as well.
\end{Lemma}

\begin{proof}
Note first that, again, $_A{\mathcal M}_B$ exists by
\cite[2.4.Lemma]{KoshitaniKunugiWaki2008}.
Then the assertion 
follows as in \cite[Proof of 6.3.Lemma]{KoshitaniKunugiWaki2008},
by applying \cite[Theorem]{KoshitaniLinckelmann2005} in order to make
use of {\bf\ref{LocalPuig}},
and using gluing through 
\cite[3.1.Theorem]{Linckelmann2001};
note that the fusion condition in the latter theorem
is automatically satisfied by \cite[1.15.Lemma]{KoshitaniKunugiWaki2004}.
\end{proof}

\begin{Notation}\label{Notation7}
We use the notation $\mathcal M$
and $\mathcal F$ as in {\bf\ref{StableEquivalenceFor2HS}}.
\end{Notation}

\begin{Lemma}\label{StableEquivalenceAandA_N}
The following holds:
\begin{enumerate}
  \renewcommand{\labelenumi}{\rm{(\roman{enumi})}}
      \item
The $(A, A_N)$-bimodule $1_A{\cdot}kG{\cdot}1_{A_N}$ has a
unique (up to isomorphism) indecomposable direct summand
${\mathcal M}_N$ with vertex $\Delta P$.
Moreover, the functor
$$ \mathcal F_N: {\mathrm{mod}}{\text{-}}A \rightarrow
   {\mathrm{mod}}{\text{-}}A_N: 
   X_A \mapsto (X \otimes_A {\mathcal M}_N)_{A_N} $$
induces a splendid stable equivalence of Morita type between
$A$ and $A_N$.
In particular, $\mathcal F$ fulfils the
assumptions of {\bf\ref{SourceVertex}}, and hence its assertions as well.
   \item
Suppose that $X$ is an indecomposable $kG$-module in $A$ such that
a vertex of $X$ belongs to ${\mathfrak A}(G,P,N)$, and let
$\mathcal F (X) = Y \oplus ({\mathrm{proj}})$
for a non-projective indecomposable $kH$-module $Y$ in $B$,
and $\mathcal F_N (X) = f(X) \oplus ({\mathrm{proj}})$,
where $f$ denotes the Green correspondence with respect to $(G,P,N)$.
Then, $f(X)$ is the correspondent of $Y$ with respect to the
Fong-Reynolds correspondence between $B$ and $A_N$, namely
$f(X)\cong Y{\uparrow}^N$.
\end{enumerate}
\end{Lemma}

\begin{proof}
(i) follows by {\bf\ref{StableEquivalenceFor2HS}}
and {\bf\ref{FongReynolds}}, and (ii) follows from {\bf\ref{SourceVertex}}.
\end{proof}

\section{Images of simples via the functor $\mathcal F_N$}\label{******}

\begin{Notation}\label{Notation5}
For brevity, let 
$F:=\mathcal F_N: \underline{\mathrm{mod}}{\text{-}}A 
    \rightarrow 
    \underline{\mathrm{mod}}{\text{-}}A_N$
denote the functor given in
{\bf\ref{StableEquivalenceAandA_N}}.
Recall that $F(X) = f(X) \oplus ({\mathrm{proj}})$,
where $f$ is the Green correspondence 
with respect to $(G, P, N)$,
whenever the Green correspondent $f(X)$ is defined,
see {\bf\ref{StableEquivalenceAandA_N}}(ii).
\end{Notation}

\begin{Lemma}\label{TrivialSourceS1S2}
The simples $S_1$ and $S_2$ are trivial source
$kG$-modules.
\end{Lemma}

\begin{proof}
By \cite[p.80]{Atlas} 
$G$ has a subgroup $U$ with $U \cong U_3(5)$.
Then, by a computation with {\sf GAP} \cite{GAP}, we know
$1_U{\uparrow}^G{\cdot}1_A = \chi_{26} + \chi_{27}$.
Therefore there is a $kG$-module $X$ such that
$X = (k_U{\uparrow}^G{\cdot}1_A)$ 
and $X$ is liftable
to an $\mathcal OG$-lattice affording $\chi_{26} + \chi_{27}$. 
Thus by {\bf\ref{DecompositionMatrixOf2HS}},
it holds that
$X = S_1 + S_2$,
as composition factors.
>From {\bf\ref{Scott}}(ii), we have
$\dim_k[{\mathrm{End}}_{kG}(X)] = 2$.
Hence, $X = S_1 \oplus S_2$ since $S_1$ is not isomorphic to $S_2$.
\end{proof}

\begin{Lemma}\label{TrivialSourceS4}
The simple $S_4$ is a trivial source
$kG$-module.
\end{Lemma}

\begin{proof}
By \cite[p.80]{Atlas} 
$G$ has a subgroup $M$ with $M \cong {\sf M}_{11}$. Then, again
by a computation with {\sf GAP} \cite{GAP}, it holds that
$1_M{\uparrow}^G{\cdot}1_A
 = \chi_{25} + \chi_{26} + \chi_{27} + \chi_{28} + \chi_{29}$.
Set $X = k_M{\uparrow}^G{\cdot}1_A$.
Thus it follows from
\cite[I Theorem 17.3]{Landrock} that $X$ has a submodule $S$
such that $S \leftrightarrow \chi_{25}$. 
By {\bf\ref{DecompositionMatrixOf2HS}}, $S = S_4$ and
$X = 2 \times S_1 + 2 \times S_2 + 2 \times S_3 + S_4$
as composition factors.
Therefore the self-duality of $S_4$ and $X$ implies
$S {\Big|} X$.
\end{proof}

\begin{Lemma}\label{TrivialSourceS5}
The simple $S_5$ is a trivial source
$kG$-module.
\end{Lemma}

\begin{proof}
As before it follows from \cite[p.80]{Atlas}
that $G$ has a 
maximal subgroup $M$ such that $M \cong 2.{\sf M}_{22}$
and $|G : M| = 100$. 
By \cite[p.39]{Atlas} and
\cite[${\sf M}_{22}$ (mod 3)]{ModularAtlas}, we know that
$M$ has a $3$-block $\tilde A$
(which is called "Block 6" in \cite[${\sf M}_{22}$ (mod 3)]{ModularAtlas})
such that
$\tilde A$ has a defect group $\tilde P$ with
$\tilde P \cong C_3 \times C_3$, where we can assume $\tilde P = P$.
Moreover, $\tilde A$ has an irreducible ordinary character
$\tilde\chi_{13}$ of degree $10$, and $\tilde A$ 
has a simple $kM$-module $\tilde S$ of dimension $10$ 
corresponding to $\tilde\chi_{13}$.
Now, it follows from \cite[Proposition 3.19]{DanzKuelshammer2009}
that $\tilde S$ 
has a trivial source.
On the other hand, a computation in {\sf GAP} \cite{GAP} shows
$\tilde\chi_{13}{\uparrow}^G = \chi_{32}$. 
Therefore {\bf\ref{DecompositionMatrixOf2HS}} yields that
$S_5\cong \tilde S{\uparrow}^G$ also has a trivial source.
\end{proof}

\begin{Lemma}\label{GreenCorrespondentF(S1)}
We can assume that
$f(S_1) = 2\alpha$ \ and \  $f(S_2) = f(S_1^*) = (2\alpha)^* = 2\beta$.
\end{Lemma}

\begin{proof}
It follows from {\bf\ref{TrivialSourceS1S2}}, 
{\bf\ref{DecompositionMatrixOf2HS}} and \cite[Lemma 2.2]{Okuyama1981} 
that $f(S_1)$ and $f(S_2)$ are simple, so 
$\{ f(S_1), f(S_2) \} \subseteq
 \{ 2\alpha, 2\beta, 2\gamma, 2\delta, 4 \}$.
Then, since $2\gamma, 2\delta, 4$ are self-dual and
$2\beta = (2\alpha)^*$ by {\bf\ref{LocalStructureI}}(v),
and since $S_2 = {S_1}^*$ by {\bf\ref{DecompositionMatrixOf2HS}},
it holds that
$\{ f(S_1), f(S_2) \} = \{ 2\alpha, 2\beta \}$.
Thus we get the assertion.
\end{proof}

\begin{Lemma}\label{GreenCorrespondentF(S4)}
We can assume that
$$ f(S_4) = 2\delta \quad\text{ and }\quad f(S_5) = 4 .$$
Recall that by considering $N$ just as an abstract group
$\{2\gamma,2\delta\}$ are indistinguishable,
see {\bf\ref{LocalStructureI}}{\rm{(v)}}.
But fixing $N\leqslant G$
and specifying $f$ serves to identify the latter uniquely.
\end{Lemma}

\begin{proof}    
It follows from {\bf\ref{TrivialSourceS4}}, {\bf\ref{TrivialSourceS5}},
{\bf\ref{GreenCorrespondentF(S1)}}, {\bf\ref{DecompositionMatrixOf2HS}}
and \cite[Lemma 2.2]{Okuyama1981} that
$ \{ f(S_4), f(S_5) \} \subseteq \{ 2\gamma, 2\delta, 4 \}$.
Now, by {\bf\ref{Notation2}} and Sylow's theorem, we have
$|G : N| \equiv 1$ (mod $3$).
Set $T_i = f(S_i)$ for $i = 4, 5$.
By the definition of Green correspondence,
${T_5}{\uparrow}^G = S_5 \oplus X$ for a $kG$-module $X$ 
such that $X$ is $Q$-projective.
Hence, by {\bf\ref{Notation2}} and 
\cite[Chap.4, Theorem 7.5]{NagaoTsushima},
it holds that
$$ \dim(S_5) \equiv \dim({T_5}{\uparrow}^G) = \dim(T_5)\cdot|G:N|
         \equiv \dim (T_5)\bmod 3 ,$$ 
and $\dim (S_5) = 1000 \equiv 1 \bmod 3 $.
Thus, $T_5 = 4$, so that $T_4 \in \{ 2\gamma, 2\delta \}$.
\end{proof}

\begin{Lemma}\label{GreenCorrespondentF(S3)}
Using the assumption of {\bf\ref{GreenCorrespondentF(S4)}},
it holds that
$$
   f(S_3) = 
  \boxed{ 
  \begin{matrix} 2\gamma \\ 4 \  \\ 2\gamma
  \end{matrix}
        }.
$$
\end{Lemma}

\begin{proof}
As noted in the proof of {\bf\ref{TrivialSourceS5}},
$G$ has a maximal subgroup $M$ such that $M \cong 2.{\sf M}_{22}$
and $|G : M| = 100$.
By \cite[p.39]{Atlas} and
\cite[${\sf M}_{22}$ (mod 3)]{ModularAtlas}, we know that
$M$ has a $3$-block $\tilde B$
(which is called "Block 7" in \cite[${\sf M}_{22}$ (mod 3)]{ModularAtlas})
such that $\tilde B$ has a defect group $\tilde Q$ with
$\tilde Q \cong C_3$, where we can assume $\tilde Q = Q$.
Moreover, $\tilde B$ has an irreducible ordinary character
$\tilde\chi_{16}$ of degree $120$, and $\tilde B$
has a simple $kM$-module $\tilde T$ of dimension $120$
corresponding to $\tilde\chi_{16}$.
Now, it follows from \cite[Proposition 3.19]{DanzKuelshammer2009}
that $\tilde T$ is a trivial source module with vertex $Q$.
Hence the indecomposable summands of $X:=\tilde T{\uparrow}^G\cdot 1_A$
have a trivial source as well, and are $Q$-projective.
On the other hand, a computation in {\sf GAP} \cite{GAP} says that
$$ X \leftrightarrow
\tilde\chi_{16}{\uparrow}^G\cdot 1_A = 
\chi_{26} + \chi_{27} + \chi_{28} + \chi_{29} .$$
Therefore, {\bf\ref{DecompositionMatrixOf2HS}} yields that
$X = 2 \times S_1 + 2 \times {S_1}^* + 2 \times S_3$
as composition factors; note that this shows that $X$ is projective-free.
Recall that $\chi_{26} \leftrightarrow S_1$ and
$\chi_{28} \leftrightarrow S_1 + S_3$. 
Hence, it holds by {\bf\ref{Scott}},
{\bf\ref{TrivialSourceS1S2}} and
{\bf\ref{DecompositionMatrixOf2HS}} that, as $k$-spaces,
$$
{\mathrm{Hom}}_{kG}(S_1, X) \cong
{\mathrm{Hom}}_{kG}({S_1}^*, X) \cong
{\mathrm{Hom}}_{kG}(X, S_1) \cong
{\mathrm{Hom}}_{kG}(X, {S_1}^*) \cong k ,
$$
$$
{\mathrm{Hom}}_{kG}(S_4, X) =
{\mathrm{Hom}}_{kG}(S_5, X) =
{\mathrm{Hom}}_{kG}(X, S_4) =
{\mathrm{Hom}}_{kG}(X, S_5) = 0 .
$$
Moreover, it follows from 
\cite[II Lemma 2.7 and Corollary 2.8]{Landrock}, 
{\bf\ref{GreenCorrespondentF(S1)}} and {\bf\ref{Notation5}} that,
as $k$-spaces,
\begin{align*}
{\mathrm{Hom}}_{kN}(F(X), 2\alpha)
\ \cong \ 
{\underline{\mathrm{Hom}}}_{kN}(F(X), 2\alpha)
\ \cong \ 
{\underline{\mathrm{Hom}}}_{kG}(X, S_1)
\ \cong \ 
{\mathrm{Hom}}_{kG}(X, S_1) \ \cong \ k.
\end{align*}
Similarly, we get 
${\mathrm{Hom}}_{kN}(F(X), (2\alpha)^*)\cong k$.
Using {\bf\ref{GreenCorrespondentF(S4)}} in the above proof,
we obtain
$$ {\mathrm{Hom}}_{kN}(F(X), 2\delta) 
 = {\mathrm{Hom}}_{kN}(F(X), 4) 
 = {\mathrm{Hom}}_{kN}(2\delta, F(X)) 
 = {\mathrm{Hom}}_{kN}(4, F(X)) 
= 0 .$$

Then, let $Y$ be an indecomposable direct summand of $X$
with $S_1{\Big|}(Y/{\mathrm{rad}}(Y))$, hence $Y$ is non-projective.
This means that we can write 
$F(Y) = U \oplus ({\mathrm{proj}})$
for a non-projective indecomposable $kN$-module $U$ in $A_N$,
where by {\bf\ref{SourceVertex}} we infer that
$U$ is a trivial source $kN$-module with vertex $Q$,
since $\tilde T$ has vertex $Q$ and $Y$ is non-projective.
Moreover, from
$$ F(X) = F(Y) \oplus ({\mathrm{module}}) 
      = U \oplus ({\mathrm{proj}}) \oplus ({\mathrm{module}}), $$
{\bf\ref{GreenCorrespondentF(S1)}} and {\bf\ref{Notation5}},
we get
${\mathrm{Hom}}_{kN}(F(Y), 2\delta) = {\mathrm{Hom}}_{kN}(F(Y), 4) = 0$
and
\begin{align*}
{\mathrm{Hom}}_{kN}(U, 2\alpha)
\ \cong \
{\underline{\mathrm{Hom}}}_{kN}(U, 2\alpha)
& \ \cong \ 
{\underline{\mathrm{Hom}}}_{kN}(F(Y), 2\alpha) \\
& \ \cong \ 
{\underline{\mathrm{Hom}}}_{kG}(Y,S_1)
\ \cong \
{\mathrm{Hom}}_{kG}(Y,S_1)
\ \cong \
k.
\end{align*}
Thus, from {\bf\ref{TrivialSourceModulesInA_N}}(iii) we conclude that
$$
U \  = \ 
 \boxed{
  \begin{matrix} 2\alpha  \ \ \ 2\gamma   \\
                 4                            \\
                 2\alpha  \ \ \ 2\gamma
  \end{matrix}
       } .
$$

Since $U{\Big|} F(X)$ and $X$ is self-dual, by 
\cite[A.2.Lemma]{KoshitaniMuellerNoeske2011}
we have $U^*{\Big|}F(X)^* = F(X^*) = F(X)$ as well. 
Since $U$ is not self-dual, this yields that
\begin{align*}
F(X) 
& = \ U \oplus U^* \oplus ({\mathrm{module}}) \oplus ({\mathrm{proj}}) \\
& = \ 
 \boxed{
  \begin{matrix} 2\alpha  \ \ \ 2\gamma   \\
                 4                            \\
                 2\alpha  \ \ \ 2\gamma
  \end{matrix}
       } 
\, \bigoplus \,
 \boxed{
  \begin{matrix} (2\alpha)^*  \ \ \ 2\gamma   \\
                   \ \ 4                            \\
                 (2\alpha)^*  \ \ \ 2\gamma
  \end{matrix}
       } 
\, \bigoplus \,
   ({\mathrm{module}}) \, \bigoplus \, ({\mathrm{proj}}),
\end{align*}
Note that since $X$ is $Q$-projective, neither $S_1$ nor
${S_1}^*$ can possibly be a direct summand of 
$X$ by {\bf\ref{DecompositionMatrixOf2HS}}(ii). 
Hence 
it follows from the {\sl stripping-off method}, see
{\bf\ref{StrippingOffMethod}}, that
there is a subquotient module $Z$ of $X$ such that 
$Z = S_3 + S_3$ as composition factors, and such that
$$
F(Z) \ = \ V \, \oplus \, V^\ast \, \oplus \, 
           ({\mathrm{module}}) \, \oplus \, ({\mathrm{proj}}),
\quad\quad {\mathrm{where}} \  
   V \, = \,  
\boxed{
\begin{matrix}
   2\gamma & {\Big|} &    &         &  \\
           & {\Big|} &  4 & {\Big|} &  \\
           &         &    & {\Big|} & 2\gamma
\end{matrix} }
.$$
Since $F$ induces a stable equivalence
by {\bf\ref{Notation5}}, we conclude that $Z$ is
decomposable, that is $Z \cong S_3 \oplus S_3$, which implies that
$ F(S_3) \oplus F(S_3) \, = \, V \oplus V^\ast $,
since $F(S_3)$ is indecomposable by
{\bf\ref{Linckelmann}}(i).
Hence we infer that $V\cong V^\ast$ is indecomposable,
having Loewy and socle series
$$ V \, = \,
   \boxed{
      \begin{matrix} 2\gamma \\ 4 \\ 2\gamma
      \end{matrix} }.
$$
Note that (although we do not need this fact) the above analysis
also shows that 
$$ X \ = \
   \boxed{
      \begin{matrix} S_1 \\ S_3 \\ S_1
      \end{matrix} }
\ \bigoplus \
   \boxed{
      \begin{matrix} S_1^* \\ S_3 \\ S_1^*
      \end{matrix} }.
$$
\end{proof}

\section{Proof of main results}\label{proof}

\begin{Lemma}\label{match2delta2d}
Keeping the setting in {\bf\ref{2HScontainsA8}} fixed, we have 
$$ 2\delta{\downarrow}^N_{H'} = 1\delta \oplus 1\delta = 1d \oplus 1d
\ \text{ and } \
2\gamma{\downarrow}^N_{H'} = 1\gamma \oplus 1\gamma = 1c \oplus 1c .$$
\end{Lemma}

\begin{proof}
By {\bf\ref{StructureOfOH2}} we have to show that $1\delta = 1d$.
In order to do so, we employ the $kG$-module $S_4$, for which
we first show that $S_4{\downarrow}_{G'} = 28 \oplus 28$:
By an explicit computation with {\sf GAP} \cite{GAP}
we know that $S_4{\downarrow}_{G'} = 2 \times 28$
as composition factors, see
{\bf\ref{DecompositionMatrixOf2HS}} and
{\bf\ref{CharactersOfA8}}.
Hence we are done as soon as we show that
${\mathrm{Ext}}^1_{kG'}(28, 28) = 0$. This in turn is seen as follows:
Since by {\bf\ref{StableEquivalenceInA8}} the functor
$\mathcal F'$ commutes with taking Heller translates
in the stable module categories, we have
$$
 {\mathrm{Ext}}_{kG'}^1(28, 28)
\    \cong \
{\mathrm{Ext}}_{kH'}^1(\mathcal F'(28), \mathcal F'(28))
\   \cong \
{\mathrm{Ext}}_{kH'}^1(1d, 1d) \ = \ 0,
$$
by making use of {\bf\ref{GreenCorrespondentA8}},
and the vanishing result in the proof of
{\bf\ref{TrivialSourceModulesInBprime}}.
(As an alternative, the vanishing of ${\mathrm{Ext}}^1_{kG'}(28, 28)$
can also be found in \cite[Appendix p.3115]{Siegel1991}.)

Now, on the one hand we get by {\bf\ref{GreenCorrespondentA8}} that
\begin{align*}
S_4{\downarrow}_{H'} 
=&
S_4{\downarrow}_{G'}{\downarrow}_{H'}
= (28 \oplus 28){\downarrow}_{H'}
\\
=& \Big(f'(28) \oplus f'(28) \oplus (\mathfrak Y(G',P,H'){\text{-proj}})\Big)
\\
=& 1d \oplus 1d \oplus (\mathfrak Y(G',P,H'){\text{-proj}}),
\end{align*}
see \cite[Chap.4, \S 4]{NagaoTsushima}.
On the other hand, we get from {\bf\ref{GreenCorrespondentF(S4)}} that 
\begin{align*}
S_4{\downarrow}_{H'}
=&
S_4{\downarrow}_{N}{\downarrow}_{H'}
= \Big( f(S_4) \oplus (\mathfrak Y(G,P,N){\text{-proj}}) \Big)
  {\downarrow}_{H'}
\\
=&
\Big( 2\delta \oplus (\mathfrak Y(G,P,N){\text{-proj}}) \Big)  
  {\downarrow}_{H'}
\\
=&
( 1\delta \oplus 1\delta)
\oplus \Big(\mathfrak Y(G,P,N){\text{-proj}} \Big){\downarrow}_{H'}.
\end{align*}
Therefore, by comparing the vertices of the indecomposables
showing up above and by Krull-Schmidt's theorem, 
we finally know that $1\delta = 1d$.
\end{proof}

\begin{Lemma}\label{PuigForAandA'}
The blocks $A$ and $A'$ are Morita equivalent
induced by an $(A,A')$-bimodule which is $\Delta P$-projective
and is a trivial source module as an $\mathcal O[G \times G']$-lattice.
\end{Lemma}

\begin{proof}
First of all, $A$ and $A'$ are splendidly stable equivalent 
of Morita type by either of the $(A,A')$-bimodules
$$ \widetilde{\mathcal M}_i \, = \,
 {_A}(\mathcal M \otimes_B B_i \otimes_{B'}{\mathcal M'}^{\vee})_{A'} ,$$
where $\mathcal M$ and $\mathcal M'$ are the same as in
{\bf\ref{StableEquivalenceFor2HS}} and
{\bf\ref{StableEquivalenceInA8}}, respectively,
and $i=1,3$ by {\bf\ref{StructureOfOH2}}(i).
Hence, the following holds from
{\bf\ref{GreenCorrespondentA8}}, {\bf\ref{StableEquivalenceAandA_N}}(ii) 
as well as 
{\bf\ref{GreenCorrespondentF(S1)}}, {\bf\ref{GreenCorrespondentF(S4)}},
{\bf\ref{GreenCorrespondentF(S3)}},
{\bf\ref{StructureOfOH2}}(ii)--(iv) and {\bf\ref{match2delta2d}}:
\begin{center}
\begin{tabular}{c c c c c c c}
mod-$A$ & $\overset{\mathcal F}{\longrightarrow}$    
        & mod-$B_i$ 
        & $\overset{{\mathrm{Res}}{\downarrow}^H_{H'}}
        {\longrightarrow}$ & mod-$B'$    
        & $\overset{{\mathcal F'}^{-1}}{\longrightarrow}$  & mod-$A'$  \\
\hline
\\
$S_1$ & $\mapsto$ & $\boxed{1\alpha_i}$ & $\mapsto$ & 
        $\left\{\begin{array}{c}\boxed{1a}\\\boxed{1b}\rule{0em}{1.2em}\\
                \end{array}\right.$
      & $\mapsto$ & 
        $\left.\begin{array}{c}k_{G'}\\7\rule{0em}{1.2em}\\
               \end{array}\right\}$ 
\\ 
\\
$S_2=S_1^*$ & 
        $\mapsto$ & $\boxed{1\beta_i}$ & $\mapsto$ & 
        $\left\{\begin{array}{c}\boxed{1b}\\\boxed{1a}\rule{0em}{1.2em}\\
                \end{array}\right.$
      & $\mapsto$ & 
        $\left.\begin{array}{c}7\\k_{G'}\rule{0em}{1.2em}\\
               \end{array}\right\}$ 
\\
\\
$S_3$ & $\mapsto$ 
      & $\boxed{\begin{matrix}1\gamma_i \\ 2_{B_i} \\ 1\gamma_i\end{matrix}}$ 
      & $\mapsto$ 
      & $\boxed{\begin{matrix} 1c \\ 2 \\ 1c \end{matrix}}$
      & $\mapsto$ & $13$          
\\ \\
$S_4$ & $\mapsto$ & $\boxed{1\delta_i}$ & $\mapsto$ & $\boxed{1d}$ 
      & $\mapsto$ & $28$          
\\ \\
$S_5$ & $\mapsto$ & $\boxed{2_{B_i}}$ & $\mapsto$ & $\boxed{2}$ 
      & $\mapsto$ & $35$      
\end{tabular}
\end{center}
Note that, by {\bf\ref{TrivialSourceModulesInBprime}}(i),
the $B'$-module $\mathcal F'(13)$ is uniquely
(up to isomorphism) determined by its Loewy series;
hence we indeed have 
$$ {\mathrm{Res}}{\downarrow}^H_{H'}(\mathcal F(S_3))
   \cong \mathcal F'(13) .$$
Therefore we finally get that
$A$ and $A'$ are Morita equivalent
by {\bf\ref{Linckelmann}}(ii).
More precisely, we know also that
the Morita equivalence is given
by either of the bimodules $\widetilde{\mathcal M}_i$,
satisfying the properties desired.
\end{proof}

\begin{Remark}\label{J4problem}
A remark on the strategy employed in the proofs of 
{\bf\ref{match2delta2d}} and {\bf\ref{PuigForAandA'}} is in order:

(a)
To derive {\bf\ref{match2delta2d}}
we use the full strength of {\bf\ref{2HScontainsA8}}:
Indeed, using an embedding $H'\leqslant G'$ and the Green
correspondence $f'$, we have defined $1d$, 
see {\bf\ref{GreenCorrespondentA8}}, and similarly, 
using an embedding $H\leqslant N\leqslant G$ and the Green
correspondence $f$, we have defined $1\delta$
see {\bf\ref{GreenCorrespondentF(S4)}}.
But from that alone we would only be able to conclude
that $\{1\gamma,1\delta\}=\{1c,1d\}$, see {\bf\ref{StructureOfOH2}}.
Now only additionally using an embedding $G'\leqslant G$, entailing
a compatible embedding $H'\leqslant H$, we are able to conclude
as in the proof of {\bf\ref{match2delta2d}}, whose starting point
is restricting $S_4$ from $G$ to $G'$.

\medskip
(b)
In order to be able to proceed as in the proof of {\bf\ref{PuigForAandA'}}
we have to ensure that the functor induced by
$\widetilde{\mathcal M}_i$ maps simple $A$-modules to
simple $A'$-modules, which happens if and only if 
$$ {\mathrm{Res}}{\downarrow}^H_{H'}(\mathcal F(S_4))
   ={\mathrm{Res}}{\downarrow}^H_{H'}(1\delta_i)=1\delta
   \overset{!}{=}1d=\mathcal F'(28) ,$$
which is proved by the full strength of {\bf\ref{match2delta2d}}.
Without using the explicit configuration of groups in
{\bf\ref{2HScontainsA8}} we only know $1\delta\in\{1c,1d\}$.
(Note that this phenomenon has also been observed in
 \cite[6.14.Question]{KoshitaniKunugiWaki2008}.)
As an alternative we would have to proceed as follows:

By {\bf\ref{CharacterTableOfBprime}}(ii)
there is an outer automorphism of $H'$,
hence inducing a Morita self-equivalence of
of $kH'=k[P \rtimes D_8]$, interchanging $1c\leftrightarrow 1d$.
Twisting the bimodule $\widetilde{\mathcal M}_i$
accordingly then still yields a Morita equivalence between $A$ and $A'$.
But the outer automorphism applied necessarily changes
the structure of $kH'$, which is its own source algebra, 
as an interior $P$-algebra; in other words, the twisted 
bimodule then no longer is $\Delta P$-projective,
hence it does no longer induce a Puig equivalence between $A$ and $A'$.
Thus, as already indicated in {\bf\ref{twoequivalences}}(a) 
we would end up with the weaker statement in 
{\bf\ref{PuigForAandA'}} only saying that $A$ and $A'$ are Morita equivalent.
\end{Remark}

\bigskip\noindent
{\bf Proof of \ref{2HSandA8}}.
By {\bf\ref{NotationEquivalences}}, the assertion 
of {\bf\ref{PuigForAandA'}} is equivalent to saying that
$A$ and $A'$ are Puig equivalent.
Moreover, the two choices of $\widetilde{\mathcal M}_i$,
for $i=1,3$, account precisely for the two bijections between
the simple $A$- and $A'$-modules as described in 
{\bf\ref{twoequivalences}}(b).
\hfill {$\square$}

\bigskip\noindent
{\bf Proof of \ref{MainTheorem}}.
This follows from {\bf\ref{2HSandA8}}, since by 
\cite[Example 4.3]{Okuyama1997} and
\cite[Theorem 3]{Okuyama2000} the block algebras $A'$ and $B'$
are splendidly Rickard equivalent,
and by {\bf\ref{StructureOfOH2}}(i) and {\bf\ref{FongReynolds}}
the block algebras $B'$ and $B$ are Puig equivalent.
\hfill {$\square$}

\bigskip\noindent
{\bf Proof of \ref{ADGCfor2HSforAllPrimes}}.
This follows from {\bf\ref{MainTheorem}} and
{\bf\ref{BrouesConjectureFor2HS}}.
\hfill {$\square$}


\bigskip

{\begin{center}
{\bf Acknowledgements}
\end{center}}

\bigskip\noindent {\rm
A part of this work was done while the first author was
staying in RWTH Aachen University in 2010, 2011 and 2012. 
He is grateful to Gerhard Hiss for his kind hospitality.
For this research the first author was partially
supported by the Japan Society for Promotion of Science (JSPS),
Grant-in-Aid for Scientific Research 
(C)20540008, 2008--2010,
(B)21340003, 2009--2011,
and also (C)23540007, 2011--2014. }

\bigskip



\begin{thebibliography}{99}

\bibitem{AlperinBroue}
J.L.~Alperin, M.~Brou{\'e},
Local methods in block theory,
Ann. of Math. {\bf 110} (1979), 143--157.

\bibitem{AlperinLinckelmannRouquier}
J.L.~Alperin, M.~Linckelmann, R.~Rouquier,
Source algebras and source modules,
J.~Algebra {\bf 239} (2001), 262--271.

\bibitem{Benson}
D.J.~Benson,
Representations and Cohomology I:
Basic representation theory of finite groups
and associative algebras,
Cambridge Univ. Press, Cambridge, 1998.

\bibitem{CTblLib}
T. Breuer, {\sf GAP4}-package CTblLib --- 
The {\sf GAP} Character Table Library, 
Version 1.1.3, 2004, 
\newline
http://www.gap-system.org/Packages/ctbllib.html.

\bibitem{Broue1990} 
M. Brou{\'e}, Isom{\'e}tries parfaites, types de blocs,
cat{\'e}gories d{\'e}riv{\'e}es, Ast\'erisque {\bf 181--182} 
(1990), 61--92.

\bibitem{BrouePuig1980}
M.~Brou{\'e}, L.~Puig,
Characters and local structure in $G$-algebras,
J.~Algebra {\bf 63} (1980), 306--317.

\bibitem{ChuangRickard}
J.~Chuang, J.~Rickard,
Representations of finite groups and tilting,
Handbook of tilting theory, pp.359--391,
London Math. Soc. Lecture Note Ser. {\bf 332}, 
Cambridge Univ. Press, Cambridge, 2007.

\bibitem{Atlas} J.H. Conway, R.T. Curtis, S.P. Norton, R.A. Parker,
R.A. Wilson, Atlas of Finite Groups,
Clarendon Press, Oxford, 1985.

\bibitem{DanzKuelshammer2009}
S.~Danz, B.~K\"ulshammer,
Vertices, sources and Green correspondents of the
simple modules for the large Mathieu groups,
J.~Algebra {\bf 322} (2009), 3919--3949.

\bibitem{GAP} The {\sf GAP} Group,
{\sf GAP} --- Groups, Algorithms, and Programming, 
Version 4.4.12, http://www.gap-system.org, 2008.

\bibitem{Huppert1966}
B.~Huppert,
Endliche Gruppen I,
Springer, Berlin, 1966.

\bibitem{ModularAtlas} C.~Jansen, K.~Lux, R.~Parker, R.~Wilson,
An Atlas of Brauer Characters, Clarendon Press, Oxford, 1995.

\bibitem{KessarLinckelmann2010}
R.~Kessar, M.~Linckelmann,
On stable equivalences and blocks with one simple module,
J.~Algebra {\bf 323} (2010), 1607--1621.

\bibitem{Kiyota1984} M. Kiyota,
On $3$-blocks with an elementary abelian defect group of order $9$,
J.~Fac. Sci. Univ. Tokyo (Section IA, Math.) {\bf 31} (1984),
33--58.

\bibitem{Knoerr} R. Kn\"orr, On the vertices of irreducible modules,
Ann. of Math. {\bf 110} (1979), 487--499.

\bibitem{Koshitani2003}
S.~Koshitani,
Conjectures of Donovan and Puig for principal $3$-blocks with
abelian defect groups,
Comm.~Algebra {\bf 31} (2003), 2229--2243;
and its Corrigendum,
Comm.~Algebra {\bf 32} (2004), 391--393.

\bibitem{KoshitaniKunugi2002} S. Koshitani, N. Kunugi,
Brou\'e's conjecture holds for principal $3$-blocks with
elementary abelian defect group of order $9$,
J.~Algebra {\bf 248} (2002), 575--604.

\bibitem{KoshitaniKunugi2005} S. Koshitani, N. Kunugi,
Blocks of central $p$-group extensions, 
Proc.~Amer.~Math. Soc. {\bf 133} (2005), 21--26. 

\bibitem{KoshitaniKunugi2010} S. Koshitani, N. Kunugi,
Trivial source modules in blocks with cyclic defect groups,
Math.~Z. {\bf 265} (2010), 161-172.

\bibitem{KoshitaniKunugiWaki2002} S. Koshitani, N. Kunugi, K. Waki,
Brou{\'e}'s conjecture for non-principal $3$-blocks of finite groups,
J. Pure Appl. Algebra {\bf 173} (2002), 177--211.

\bibitem{KoshitaniKunugiWaki2004}
S.~Koshitani, N.~Kunugi, K.~Waki, 
Brou{\'e}'s abelian defect group conjecture for the Held 
group and the sporadic Suzuki group,
J.~Algebra {\bf 279} (2004), 638--666.

\bibitem{KoshitaniKunugiWaki2008}
S.~Koshitani, N.~Kunugi, K.~Waki, 
Brou{\'e}'s abelian defect group conjecture holds
for the Janko simple group $J_4$,
J.~Pure Appl.~Algebra {\bf 212} (2008), 1438--1456.

\bibitem{KoshitaniLinckelmann2005} S. Koshitani, M. Linckelmann,
The indecomposability of a certain bimodule given by the 
Brauer construction, J.~Algebra {\bf 285} (2005), 726--729.

\bibitem{KoshitaniMueller2010}
S.~Koshitani, J.~M{\"u}ller,
Brou{\'e}'s abelian defect group conjecture holds
for the Harada-Norton sporadic simple group HN,
J.~Algebra {\bf 324} (2010), 394--429.

\bibitem{KoshitaniMuellerNoeske2011}
S.~Koshitani, J.~M{\"u}ller, F.~Noeske,
Brou{\'e}'s abelian defect group conjecture holds
for the sporadic simple Conway group {\sf Co}$_3$,
J.~Algebra {\bf 348} (2011), 354--380.

\bibitem{Kunugi} N. Kunugi, Morita equivalent $3$-blocks of the
$3$-dimensional projective special linear groups,
Proc.~London Math.~Soc. {\bf (3) 80} (2000), 575--589.

\bibitem{Landrock} P. Landrock, Finite Group Algebras and Their Modules,
London Math.~ Society Lecture Note Series, 
Vol.{\bf 84}, London Math. Soc., Cambridge, 1983.

\bibitem{Linckelmann1991} M. Linckelmann, 
Derived equivalence for cyclic blocks over a $P$-adic ring,
Math.~Z. {\bf 207} (1991), 293--304.

\bibitem{Linckelmann1996MathZ} M. Linckelmann, 
Stable equivalences of Morita type
for self-injective algebras and $p$-groups, 
Math.~Z. {\bf 223} (1996),
87--100.

\bibitem{Linckelmann1998} M. Linckelmann, 
On derived equivalences and local structure of
blocks of finite groups,
Turkish J. Math. {\bf 22} (1998), 93--107.

\bibitem{Linckelmann2001}
M.~Linckelmann,
On splendid derived and stable equivalences between
blocks of finite groups,
J.~Algebra {\bf 242} (2001), 819--843.

\bibitem{LuxMueRin} K. Lux, J. M\"uller, M. Ringe,
   Peakword condensation and submodule lattices:
   an application of the {\sf MeatAxe},
   J. Symb. Comput. {\bf 17} (1994), 529--544.

\bibitem{LuxSzokeII} K. Lux, M. Sz\H{o}ke,
Computing decompositions of modules over finite-dimensional algebras,
Experiment. Math. {\bf 16} (2007), 1--6.

\bibitem{LuxSzoke} K. Lux, M. Sz\H{o}ke,
   Computing homomorphism spaces between modules over finite
   dimensional algebras,
   Experiment. Math. {\bf 12} (2003), 91--98.

\bibitem{LuxWie} K. Lux, M. Wiegelmann,
   Determination of socle series using the condensation method, in:
   Computational algebra and number theory, Milwaukee, 1996,
   J. Symb. Comput. {\bf 31} (2001), 163--178.

\bibitem{Morita1951}
K.~Morita,
On group rings over a modular field which possess radicals 
expressible as principal ideals,
Science Report of Tokyo Bunrika Daigaku
{\bf A4} (1951), 177--194.

\bibitem{MueLatt} J. M\"uller,
   On a theorem by Benson and Conway,
   J. Pure Appl. Algebra {\bf 208} (2007), 89--100.

\bibitem{MuellerSchaps2008}
J.~M{\"u}ller, M.~Schaps,
The Brou{\'e} conjecture for the faithful $3$-blocks
of $4.M_{22}$,
J.~Algebra {\bf 319} (2008), 3588--3602.

\bibitem{NagaoTsushima}
H.~Nagao, Y.~Tsushima,
Representations of Finite Groups,
Academic Press, New York, 1988.

\bibitem{Okuyama1981} T. Okuyama, Module correspondence in finite groups,
Hokkaido Math.~J. {\bf 10} (1981), 299--318.

\bibitem{Okuyama1997} T. Okuyama, Some examples of derived equivalent 
blocks of finite groups, preprint (1997).

\bibitem{Okuyama2000} T. Okuyama, Remarks on splendid tilting complexes,
in: Representation theory of finite groups and related topics,
edited by S. Koshitani, RIMS Kokyuroku {\bf 1149}, 
Proc. Research Institute for Mathematical Sciences, 
Kyoto University, 2000, 53--59.

\bibitem{Puig1988JAlg} L. Puig, Pointed groups and construction
of modules, J.~Algebra {\bf 116} (1988), 7--129.

\bibitem{Puig1999} L. Puig, On the Local Structure of Morita and Rickard
Equivalences between Brauer Blocks,
Birkh{\"a}user Verlag, Basel, 1999.

\bibitem{Rickard1989} J.~Rickard, 
Derived categories and stable equivalence,
J.~Pure Appl.~ Algebra {\bf 61} (1989), 303--317.

\bibitem{Rickard1996} 
J. Rickard, Splendid equivalences: derived categories
and permutation modules, Proc. London Math. Soc. {\bf (3) 72} (1996),
331--358.

\bibitem{Rickard1998} 
J.~Rickard, Triangulated categories in the modular
representation theory of finite groups, 
in: Derived Equivalences for Group Rings, 
edited by S.~K{\"o}nig and A.~Zimmermann, 
Lecture Notes in Math., Vol.{\bf 1685},
Springer, Berlin, 1998, pp.177--198. 

\bibitem{MA} M. Ringe,
The {\sf C-MeatAxe}, Version 2.4, 
http://www.math.rwth-aachen.de/homes/MTX, 1998.

\bibitem{Rouquier1995} 
R.~Rouquier,
\mbox{}From stable equivalences to Rickard equivalences for
blocks with cyclic defect,
in: Groups '93 Galway/St Andrews Vol.2, 
edited by C.M.~Campbell et al.,
London Math.~Society Lecture Note Series,
Vol.{\bf 212} (1995), pp.512--523,
London Math. Soc., Cambridge.

\bibitem{Rouquier1998} R. Rouquier, 
The derived category of blocks with cyclic defect groups,
in: Derived Equivalences for Group Rings, 
edited by S.~K{\"o}nig and A.~Zimmermann, 
Lecture Notes in Mathematics, Vol.{\bf 1685},
Springer, Berlin, 1998, pp.199--220. 

\bibitem{Siegel1991}
S.~Siegel, Projective modules for $A_9$ in characteristic three,
Comm.~Alg.~{\bf 19} (1991), 3099--3117.

\bibitem{Thevenaz} 
J.~Th{\'e}venaz,
$G$-Algebras and Modular Representation Theory,
Clarendon Press, Oxford, 1995.

\bibitem{Waki1989} K. Waki,
On ring theoretical structure of $3$-blocks of
finite groups with elementary abelian defect groups
of order 9 (written in Japanese),
Master Thesis, Chiba University Japan, 1989. 

\bibitem{Wilson} R. Wilson,
   Standard generators for sporadic simple groups,
   J. Algebra {\bf 184} (1996), 505--515.

\bibitem{AtlasRep}
R. Wilson, R. Parker, S. Nickerson, J. Bray, T. Breuer,
{\sf GAP4}-package AtlasRep --- A {\sf GAP} Interface to the 
Atlas of Group Representations, Version 1.4.0, 
http://www.gap-system.org/Packages/atlasrep.html, 2008.

\bibitem{ModularAtlasProject}
R.~Wilson, J.~Thackray, R.~Parker, F.~Noeske,
J.~M{\"u}ller, F.~L{\"u}beck, C.~Jansen, G.~Hiss, T.~Breuer,
The Modular Atlas Project,
http://www.math.rwth-aachen.de/$^{\sim}$MOC.

\bibitem{ModAtlasRep}
R. Wilson, P. Walsh, J. Tripp, I. Suleiman, R. Parker, S. Norton,
S. Nickerson, S. Linton, J. Bray, R. Abbott,
Atlas of Finite Group Representations,
http://brauer.maths.qmul.ac.uk/Atlas/v3.

\end{thebibliography}
\end{document}